\documentclass{amsart}
\usepackage{latexsym}
\usepackage{amssymb}
\usepackage{amsfonts}
\usepackage{graphicx}
\usepackage{epsf}
\usepackage[all]{xypic}
\usepackage{delarray}
\usepackage{setspace}
\newtheorem{theorem}{Theorem}[section]
\newtheorem{corollary}[theorem]{Corollary}
\newtheorem{lemma}[theorem]{Lemma}
\newtheorem{proposition}[theorem]{Proposition}

\newtheorem{example}[theorem]{Example}
\newtheorem{remark}[theorem]{Remark}

\definemorphism{dato}\dashed \tip \notip
\definemorphism{Dato}\Dashed \tip \notip

\newcommand{\ra}{\rightarrow}

\newcommand{\Ff}{\mathcal{F}}

\newcommand{\Pp}{\mathcal{P}}

\def\AA{{\mathbb A}}

\def\NN{{\mathbb N}}
\def\CC{{\mathbb C}}
\def\ZZ{{\mathbb Z}}

\begin{document}

\sloppy

\title[Path, Incidence coalgebras and quantum groups]{Path subcoalgebras,  finiteness properties and quantum groups}

\subjclass[2010]{16T15, 16T05, 05C38, 06A11, 16T30}
\keywords{incidence coalgebra, path coalgebra, co-Frobenius
coalgebra, quasi-co-Frobenius coalgebra, balanced bilinear form,
 quantum group, integral}

\begin{abstract}
We study subcoalgebras of path coalgebras that are spanned by paths
(called path subcoalgebras) and subcoalgebras of incidence
coalgebras, and propose a unifying approach for these classes. We
discuss the left quasi-co-Frobenius  and the left co-Frobenius
properties for these coalgebras. We classify the left co-Frobenius
path subcoalgebras, showing that they are direct sums of certain
path subcoalgebras arising from the infinite line quiver or from
cyclic quivers. We investigate which of the co-Frobenius path
subcoalgebras can be endowed with Hopf algebra structures, in order
to produce some quantum groups with non-zero integrals, and we
classify all these structures over a field with primitive roots of
unity of any order. These turn out to be liftings of quantum lines
over certain not necessarily abelian groups.
\end{abstract}

\author{S.D\u{a}sc\u{a}lescu${}^{1,*}$, M.C. Iovanov${}^{1,2}$, C. N\u{a}st\u{a}sescu${}^1$}
\address{${}^1$University of Bucharest, Facultatea de Matematica si Informatica\\ Str.
Academiei 14, Bucharest 1, RO-010014, Romania}
\address{${}^2$University of Southern California, 3620 S Vermont Ave, KAP 108, Los Angeles, CA 90089, USA;} 
\address{e-mail: sdascal@fmi.unibuc.ro, yovanov@gmail.com,Constantin\_nastasescu@yahoo.com} 
\thanks{${}^*$ corresponding author}

\date{}

\maketitle

\section{Introduction and Preliminaries}

Let $K$ be an arbitrary field. A quadratic algebra  is a quotient of
a free noncommutative algebra $K<x_1,\dots,x_n>$ in $n$ variables by
an ideal $I$ generated by elements of degree $2$. The usual
commutative polynomial ring is such an example, with $I$ generated
by $x_ix_j-x_jx_i$. Quadratic algebras are important in many places
in mathematics, and one relevant class of such objects consists of
Koszul algebras and Koszul duals of  quadratic algebras. More
generally, one can consider quotients $K<x_1,\ldots ,x_n>/I$ for
ideals $I$ generated by homogeneous elements. Several algebras occur
in this way in topology, noncommutative geometry, representation
theory, or theoretical physics (see the examples and references in
\cite{berger2}). Such are the cubic Artin-Schreier regular algebras
$\CC<x,y>/(ay^2x+byxy+axy^2+cx^3,ax^2y+bxyx+ayx^2+xy^3)$ in
noncommutative projective algebraic geometry (see \cite{artins}),
the skew-symmetrizer killing algebras
$\CC<x_1,\dots,x_n>/(\sum\limits_{\sigma\in\Sigma_p}{\rm
sgn}(\sigma)x_{i_{\sigma(1)}}\dots x_{i_{\sigma(p)}})$ (the ideal we
factor by has $n\choose p$ generators, each one corresponding to
some fixed $1\leq i_1<\ldots <i_p\leq n$) for a fixed $2\leq p\leq
n$, in representation theory (see \cite{berger}), or the Yang-Mills
algebras
$\CC<\nabla_0,\dots,\nabla_n>/(\sum_{\lambda,\mu}g^{(\lambda,\mu)}[\nabla_\lambda[\nabla_\mu,\nabla_\nu]])$
(with $(g^{(\lambda,\mu)})_{\lambda,\mu}$ an invertible symmetric
real matrix, and the ideal we factor by has $n+1$ generators, as
$0\leq \nu\leq n$) in theoretical physics (see \cite{cdb}), to name
a few. More generally, one could start with a quiver $\Gamma$, and
define path algebras with relations by taking quotients of the path
algebra $K[\Gamma]$ by an ideal (usually) generated by homogeneous
elements, which are obtained as linear combinations of paths of the
same length. Note that the examples above are of this type: the free
algebra with $n$ elements can be thought as the path algebra of the
quiver $\Gamma$ with one vertex $1$ (which becomes the unit in the
algebra) and $n$ arrows $x_1,\dots,x_n$ starting and ending at $1$;
the relations are then given by linear combinations of paths of the
same length. This approach, for example, allows the generalization
of N-Koszulity to quiver algebras with relations, see \cite{gmmz}.

\vspace{.5cm}

We aim to study a general situation which is dual to the ones above,
but is also directly connected to it. If $\Gamma$ is a quiver, the
path algebra $K[\Gamma]$ of $\Gamma$ plays an important role in the
representation theory of $\Gamma$. The underlying vector space of
the path algebra also has a coalgebra structure, which we denote by
$K\Gamma$ and call the path coalgebra of $\Gamma$. One motivation
for replacing path algebras by path coalgebras is the following:
given an algebra $A$, and its  category of  finite dimensional
representations, one is often lead to considering the category
${{Ind}}(A)$ generated by all these finite dimensional
representations (direct limits of finite dimensional
representations). ${{Ind}}(A)$ is well understood as the category of
comodules over  the finite dual coalgebra $A^0$ of $A$ (also called
the algebra of representative functions on $A$), and it cannot be
regarded as a full category of modules over a ring unless $A$ is
finite dimensional. Such situations extend beyond the realm of pure
algebra, encompassing representations of compact groups, affine
algebraic groups or group schemes, differential affine groups, Lie
algebras and Lie groups, infinite tensor categories etc.

\vspace{.5cm}

Another reason for which the study of path coalgebras is interesting
is that any pointed coalgebra embeds into the path coalgebra of the
associated Gabriel quiver, see \cite{ni}, \cite{cm}. On the other hand, 
if $X$ is a locally finite partially
ordered set, the incidence coalgebra $KX$ provides a good framework
for interpreting several combinatorial problems in terms of
coalgebras, as explained by Joni and Rota in \cite{jr}. There are
several features common to path coalgebras and incidence coalgebras.
They are both pointed, the group-like elements recover the vertices
of the quiver, respectively the points of the ordered set, the
injective envelopes of the simple comodules have similar
descriptions, etc. Moreover, as we show later in Section \ref{snew},
Proposition \ref{propembedding}, any incidence coalgebra  embeds in
a path coalgebra, and in many situations, it has a basis where each
element is a sum of paths of the same length. We note that this is
precisely the dual situation to that considered above: for algebras,
one considers a path algebra with homogeneous relations, that is
$K[\Gamma]$ quotient out by an ideal generated by homogeneous
elements, i.e. sums of paths of the same length, with coefficients.
For a coalgebra, one considers subcoalgebras of the path coalgebra
of $\Gamma$ such that the coalgebra has a basis consisting of linear
combinations of paths of the same length (``homogeneous'' elements;
more generally, a coalgebra generated by such elements).

\vspace{.5cm}

In this paper we study Frobenius type properties for path
coalgebras, incidence coalgebras and certain subcoalgebras of them.
Recall that a coalgebra $C$ is called left co-Frobenius if $C$
embeds in $C^*$ as a left $C^*$-module. Also, $C$ is called left
quasi-co-Frobenius if $C$ embeds in a free module as a left
$C^*$-module. The (quasi)-co-Frobenius properties are interesting
for at least three reasons. Firstly, coalgebras with such properties
have rich representation theories. Secondly, for a Hopf algebra $H$,
it is true that $H$ is left quasi-co-Frobenius if and only if $H$ is
left co-Frobenius, and this is also equivalent to $H$ having
non-zero left (or right) integrals. Co-Frobenius Hopf algebras are
important since they generalize the algebra of representative
functions $R(G)$ on a compact group $G$, which is a Hopf algebra
whose integral is the left Haar integral of $G$.  Moreover, more
recent generalizations of these have been made to compact and
locally compact quantum groups (whose representation categories are
not necessarily semisimple). Thus co-Frobenius coalgebras may be the
underlying coalgebras for interesting quantum groups with non-zero
integrals. Thirdly, by keeping in mind the duality with Frobenius
algebras in the finite dimensional case, co-Frobenius coalgebras
have connections to topological quantum field theory.

\vspace{.5cm}

We propose an approach leading to similar results for path
coalgebras and incidence coalgebras, and which also points out the
similarities between these as mentioned above. It will follow from
our results that a path coalgebra (or an incidence coalgebra) is
left (quasi)-co-Frobenius if and only if the quiver consists only of
isolated points, i.e. the quiver does not have arrows (respectively
the order relation is the equality). Thus the left co-Frobenius
coalgebras arising from path coalgebras or incidence coalgebras are
just grouplike coalgebras. In order to discover more interesting
left co-Frobenius coalgebras, we focus our attention to classes of
coalgebras larger than just path coalgebras and incidence
coalgebras. On one hand we consider subcoalgebras of path coalgebras
which have a linear basis consisting of paths. We call these {\it
path subcoalgebras}. On the other hand, we look at subcoalgebras of
incidence coalgebras; any such coalgebra has a basis consisting of
segments.  In Section \ref{sectionbilforms} we apply a classical
approach to the (quasi)-co-Frobenius property. It is known that a
coalgebra $C$ is left co-Frobenius if and only if there exists a
left non-degenerate $C^*$-balanced bilinear form on $C$. Also, $C$
is left quasi-co-Frobenius if and only if there exists a family
$(\beta_i)_{i\in I}$ of $C^*$-balanced bilinear forms on $C$ such
that for any non-zero $x\in C$ there is $i\in I$ with
$\beta_i(x,C)\neq 0$.  We describe the balanced bilinear forms on
path subcoalgebras and subcoalgebras of incidence coalgebras. Such a
description was given in \cite{DNV} for the full incidence
coalgebra, and in \cite{br} for certain matrix-like coalgebras. In
Section \ref{s3} we use this description and an approach using the
injective envelopes of the simple comodules to show that a coalgebra
lying in one of the two classes is left quasi-co-Frobenius if and
only if it is left co-Frobenius, and to give several equivalent
conditions including combinatorial ones (just in terms of paths of
the quiver, or segments of the ordered set).

In Section \ref{s4} we classify all possible left co-Frobenius path
subcoalgebras. We construct some classes of left co-Frobenius
coalgebras $K[\AA_{\infty},r]$ and $K[\AA_{0,\infty},r]$ starting
from the infinite line quiver $\AA_\infty$, and a class of left
co-Frobenius coalgebras $K[\CC_n,s]$ starting from cyclic quiver
$\CC_n$. Our result says that any left co-Frobenius path
subcoalgebra is isomorphic to a direct sum of coalgebras of types
$K[\AA_{\infty},r]$, $K[\AA_{0,\infty},r]$, $K[\CC_n,s]$ or $K$,
with special quivers $\AA_{\infty}, \AA_{0,\infty},\CC_n$ and $r,s$
being certain general types of functions on these quivers. For
subcoalgebras of incidence coalgebras we do not have a complete
classification in the left co-Frobenius case. We show in Section
\ref{snew} that more complicated examples than the ones in the path
subcoalgebra case can occur for subcoalgebras of incidence
coalgebras, and a much larger class of such coalgebras is to be
expected. Also, we give several examples of co-Frobenius
subcoalgebras of path coalgebras, which are not path subcoalgebras,
and moreover, examples of pointed co-Frobenius coalgebras which are
not isomorphic to any one of the above mentioned classes.  In
Section \ref{s6} we discuss the possibility of defining Hopf algebra
structures on the path subcoalgebras that are left and right
co-Frobenius, classified in Section \ref{s4}. The main reason for
asking this question is the interest in constructing quantum groups
with non-zero integrals, whose underlying coalgebras are path
subcoalgebras. We answer completely this question  in the case where
$K$ contains primitive roots of unity of any positive order. Thus we
determine all possible co-Frobenius path subcoalgebras admitting a
Hopf algebra structure. Moreover, we describe up to an isomorphism
all such Hopf algebra structures. It turns out that they are
liftings of quantum lines over certain not necessarily abelian
groups. In particular, this also answers the question of finding the
Hopf algebra structures on finite dimensional path subcoalgebras and
on quotients of finite dimensional path algebras by ideals spanned
by paths. Our results contain, as particular cases, some results of
\cite{chyz}, where finite quivers $\Gamma$ and finite dimensional
path subcoalgebras $C$ of $K\Gamma$ are considered, such that $C$
contains all vertices and arrows of $\Gamma$. The co-Frobenius
coalgebras of this type are determined, and  all Hopf algebra
structures on them are described in \cite{chyz}. These results
follow from our more general Theorem \ref{th.qcf} and Theorem
\ref{teoremastructuriHopf}. We note that Hopf algebra structures on
incidence coalgebras have been of great interest for combinatorics,
see for example \cite{sch}, \cite{af}. We also note that the
classification of path coalgebras that admit a graded Hopf algebra
structure was done in \cite{cr}, see also \cite{gs} for a different
point of view on Hopf algebra structures on path algebras. In
particular, some of the examples in the classification have deep
connections with homological algebra: the monoidal category of chain
$s$-complexes of vector spaces over $K$ is monoidal equivalent to
the category of comodules of $K[\AA_\infty|s]$, a subclass of the
 Hopf algebras classified here (\cite{IG, B}).

\vspace{.5cm}

We also note that the unifying approach we propose here seems to
suggest that in general for pointed coalgebras interesting methods
and results could be obtained provided one can find some suitable
bases with properties resembling those of paths in quiver algebras
or segments in incidence coalgebras.

\vspace{.5cm}

Throughout the paper $\Gamma=(\Gamma_0,\Gamma_1)$ will be a quiver.
$\Gamma_0$ is the set of vertices, and $\Gamma_1$ is the set of
arrows of $\Gamma$. If $a$ is an arrow from the vertex $u$ to the
vertex $v$, we denote $s(a)=u$ and $t(a)=v$. A path in $\Gamma$ is a
finite sequence of arrows $p=a_1a_2\ldots a_n$, where $n\geq 1$,
such that $t(a_i)=s(a_{i+1})$ for any $1\leq i\leq n-1$. We will
write $s(p)=s(a_1)$ and $t(p)=t(a_n)$. Also the length of such a $p$
is ${\rm length}(p)=n$. Vertices $v$ in $\Gamma_0$ are also
considered as paths of length zero, and we write $s(v)=t(v)=v$. If
$q$ and $p$ are two paths such that $t(q)=s(p)$, we consider the
path $qp$ by taking the arrows of $q$ followed by the arrows of $p$.
We denote by $K\Gamma$ the path coalgebra, which is the vector space
with a basis consisting of all paths in $\Gamma$, and
comultiplication $\Delta$  defined by $\Delta(p)=\sum
_{qr=p}q\otimes r$ for any path $p$, and  counit $\epsilon$ defined
by $\epsilon(v)=1$ for any vertex $v$, and $\epsilon(p)=0$ for any
path of positive length. In particular, the arrows $x$ between two
vertices $v$ and $w$, i.e. $s(x)=v, t(x)=w$, are the nontrivial
elements of $P_{w,v}$, the space of $(w,v)$-skew-primitive elements:
$\Delta(x)=v\otimes x + x\otimes w$. When we use Sweedler's sigma
notation $\Delta(p)=\sum p_1\otimes p_2$ for a path $p$, we always
take representations of the sum such that all
$p_1$'s and $p_2$'s are paths.\\
We also consider  partially ordered sets $(X,\leq)$ which are
locally finite, i.e. the interval $[x,y]=\{z|\; x\leq z\leq y\}$ is
finite for any $x\leq y$. The incidence $K$-coalgebra of $X$,
denoted by $KX$, is the $K$-vector space with basis $\{
e_{x,y}|x,y\in X, x\leq y\}$, and comultiplication $\Delta$ and
counit $\epsilon$ defined by
$$\Delta(e_{x,y})=\sum_{x\leq z\leq y}e_{x,z}\otimes e_{z,y}$$
$$\epsilon (e_{x,y})=\delta_{x,y}$$
for any $x,y\in X$ with $x\leq y$, where by $\delta_{x,y}$ we denote
Kronecker's delta. The elements $e_{x,y}$ are called segments.
Again, when we use Sweedler's sigma notation $\Delta(p)=\sum
p_1\otimes p_2$ for a segment $p$, we always take representations of
the sum such that all $p_1$'s and $p_2$'s are segments. Recall that
the length of a segment $e_{x,y}$ is the maximum length $n$ of a
chain $x=z_0<z_1<\dots<z_n=y$\\ For basic terminology and notation
about coalgebras and Hopf algebras we refer to \cite{DNR} and
\cite{mo}.

\section{Balanced bilinear forms for path subcoalgebras and for subcoalgebras of incidence
coalgebras}\label{sectionbilforms}

In the rest of the paper we will be interested in two classes of
coalgebras more general than path coalgebras and incidence
coalgebras. Thus we will study

\vspace{.4cm}

$\bullet$ Subcoalgebras of the path coalgebra $K\Gamma$ having a
basis $\mathcal{B}$ consisting of paths in $\Gamma$. Such a
coalgebra will be called a path subcoalgebra. Note that if $p\in
\mathcal{B}$, then any subpath of $p$, in particular any vertex
involved in $p$, lies in $\mathcal{B}$.\\
$\bullet$ Subcoalgebras of the incidence coalgebra $KX$. By
\cite[Proposition 1.1]{DNV}, any such subcoalgebra has a basis
$\mathcal{B}$ consisting of segments $e_{x,y}$, and moreover, if
$e_{x,y}\in \mathcal{B}$ and $x\leq a\leq b\leq y$, then $e_{a,b}\in
\mathcal{B}$.

\vspace{.4cm}

It is clear that for a coalgebra $C$ of one of these two types, the
distinguished basis $\mathcal{B}$ consists of all paths (or
segments) which are elements of $C$. Let $C$ be a coalgebra of one
of these two types, with basis $\mathcal{B}$ as above. When we use
Sweedler's sigma notation $\Delta (p)=\sum p_1\otimes p_2$ for $p\in
\mathcal{B}$, we always consider representations of the sum such
that all $p_1$'s and $p_2$'s are in $\mathcal{B}$.

A bilinear form $\beta:C\times C\rightarrow K$ is $C^*$-balanced if
\begin{equation}\label{eq*}
\sum \beta(p_2,q)p_1=\sum \beta(p,q_1)q_2 \;\;\;\;\;\; {\rm for}\;
{\rm any}\; p,q\in \mathcal{B}
\end{equation}
It is clear that (\ref{eq*}) is equivalent to the fact that for any
$p,q\in \mathcal{B}$, the following three conditions hold.

\begin{eqnarray}
\beta(p_2,q)& = & \beta(p,q_1)\,\,\,\,\,{\rm for\,those\,of\,the\,}p_2{\rm's\,and\,the\,}q_1{\rm's\,such\,that\,}p_1=q_2\label{eq.d1}\\
\beta(p_2,q) & = & 0\,\,\,\,\,{\rm for\,those\,}p_2{\rm 's\,for\,which\,}p_1{\rm\,is\,not\,equal\,to\,any\,}q_2\label{eq.d2}\\
\beta(p,q_1) & = & 0\,\,\,\,\,{\rm for\,those\,}q_1{\rm
's\,for\,which\,}q_2{\rm
\,is\,not\,equal\,to\,any\,}p_1\label{eq.d3}
\end{eqnarray}

In the following two subsections we discuss separately path
subcoalgebras and subcoalgebras of incidence coalgebras.

\subsection{Path subcoalgebras}

In this subsection we consider the case where $C$ is a path
subcoalgebra. We note that if $\Gamma$ is acyclic, then for any
paths $p$ and $q$ there is at most a pair $(p_1,q_2)$ (in
(\ref{eq*})) such
that $p_1=q_2$. \\
Denote by $\Ff$ the set of all paths $d$ satisfying the following
three properties\\
$\bullet$ $d=qp$ for some $q,p\in \mathcal{B}$.\\
$\bullet$ For any representation $d=qp$ with $q,p\in \mathcal{B}$,
and any arrow $a\in\Gamma_1$, if $ap\in \mathcal{B}$ then $q$ must
end with
$a$.\\
$\bullet$ For any representation $d=qp$ with $q,p\in \mathcal{B}$,
and any arrow $b\in\Gamma_1$, if $qb\in \mathcal{B}$ then $p$ starts
with $b$.

\vspace{.4cm}

Now we are able to describe all balanced bilinear forms on $C$.

\begin{theorem} \label{formebilpath}
A bilinear form $\beta:C\times C\rightarrow K$ is $C^*$-balanced if
and only if there is a family of scalars $(\alpha_d)_{d\in\Ff}$ such
that for any $p,q\in \mathcal{B}$
\begin{eqnarray*}
\beta(p,q) & = & \left\{
\begin{array}{l}
    \alpha_d,\,\,\,{\rm if\,}s(p)=t(q){\rm\,and\,}qp=d\in\Ff\\
    0,\,\,\,\;\;{\rm otherwise}
\end{array}\right.
\end{eqnarray*}
In particular the set of all $C^*$-balanced bilinear forms on $C$ is
in bijective correspondence to $K^\Ff$.
\end{theorem}
\begin{proof}
Assume that $\beta$ is $C^*$-balanced. If $p,q\in \mathcal{B}$ and
$t(q)\neq s(p)$, then $\beta(p,q)s(p)$ appears in the left-hand side
of (\ref{eq*}), but $s(p)$ does not show up in the right-hand side,
so $\beta(p,q)=0$.
Let $\Pp$ be the set of all paths in $\Gamma$ for which there are
$p,q\in \mathcal{B}$ such that $d=qp$. Let $d\in \Pp$ and let
$d=qp=q'p'$, $p,q,p',q'\in \mathcal{B}$ be two different
decompositions of $d$, and say that, for example, ${\rm
length}(p')<{\rm length}(p)$. Then there is a path $r$ such that
$p=rp'$ and $q'=qr$, and clearly $r\in \mathcal{B}$ since it is a
subpath of $q'\in B$. Use (\ref{eq.d1}) for $p$ and $q'$, for which
there is an equality $p_1=q'_2=r$ (and the corresponding $p_2=p'$
and $q'_1=q$), and find that $\beta(p',q')=\beta(p,q)$. Therefore,
for any $d\in \Pp$ (not necessarily in $\mathcal{B}$) and any
$p,q\in \mathcal{B}$ such that $d=qp$, the scalar $\beta(p,q)$
depends only on $d$. This shows that there is a family of scalars
$(\alpha_d)_{d\in \Pp}$ such that $\beta(p,q)=\alpha_d$ for any
$p,q\in \mathcal{B}$
with $qp=d$.\\
Let $d\in\Pp$ such that  $d=qp$ for some $p,q\in \mathcal{B}$, and
there is an arrow $a\in\Gamma_1$ with $ap\in \mathcal{B}$, but $q$
does not end with $a$. That is, $q$ is not of the form $q=ra$ for
some path $r\in \mathcal{B}$. We use (\ref{eq.d2}) for the paths
$ap\in \mathcal{B}$ and $q\in \mathcal{B}$, more precisely, for the
term $(ap)_1=a$, which cannot be equal to any of the $q_2$'s
(otherwise $q$ would end with $a$), and we see that
$\beta(p,q)=0$, i.e. $\alpha_d=0$.\\
Similarly, if $d\in\Pp$, $d=qp$ with $p,q\in \mathcal{B}$ and there
is
 $b\in \Gamma_1$ with $qb\in B$ and $p$ not of the form $br$ for some path $r$
 (i.e. $p$ does not start with $b$), then we use (\ref{eq.d3}) for $p$ and $qb$,
 and $(qb)_2=b$, and we find that $\beta(p,q)=0$, i.e. $\alpha_d=0$. In conclusion,
 $\alpha_d$ may be non-zero only for $d\in\Ff$.\\
Conversely, assume that $\beta$ is of the form indicated in the
statement. We show that (\ref{eq.d1}), (\ref{eq.d2}) and
(\ref{eq.d3}) are satisfied. Let $p,q\in \mathcal{B}$ be such that
$p_1=q_2=r$ for some $p_1$ and $q_2$ (from the comultiplication
$\sum p_1\otimes p_2$ of $p$ and, respectively, the comultiplication
$\sum q_1\otimes q_2$ of $q$). Then $p=rp'$ and $q=q'r$ for some
$p',q'\in \mathcal{B}$. Let $d=q'rp'$. If $d\in\Ff$, then
$\beta(p',q)=\beta(p,q')=\alpha_d$, while if $d\notin\Ff$ we have
that $\beta(p',q)=\beta(p,q')=0$ by definition. Thus (\ref{eq.d1})
holds. Now let $p,q\in \mathcal{B}$ and fix some $p_2$ (from the
comultiplication $\sum p_1\otimes p_2$ of $p$) such that the
corresponding $p_1$ is not equal to any $q_2$. If $s(p_2)\neq t(q)$,
then clearly $\beta(p_2,q_1)=0$ by the definition of $\beta$. If
$s(p_2)=t(q)$, then $d=qp_2\notin\Ff$. Indeed, let $r$ be a maximal
path such that $p_1=er$ for some path $e$ and $q$ ends with $r$, say
$q=q'r$. Note that $e$ has length at least $1$, since $p_1$ is not
equal to any of the $q_2$'s. Then the terminal arrow of $e$ cannot
be the terminal arrow of $q'$, and this shows that
$d=p_2q=(p_2r)q'\notin\Ff$. Then $\beta(p_2,q)=0$ and (\ref{eq.d2})
is satisfied. Similarly, (\ref{eq.d3}) is satisfied.
\end{proof}

\subsection{Subcoalgebras of incidence coalgebras}

In this subsection we assume that $C$ is a subcoalgebra of the
incidence coalgebra $KX$. Let $\mathcal{D}$ be the set of all pairs
$(x,y)$ of elements in $X$ such that $x\leq y$ and there exists $x'$
with $x\leq x'\leq y$ and $e_{x,x'},e_{x',y}\in \mathcal{B}$. Fix
$(x,y)\in \mathcal{D}$. Let
$$U_{x,y}=\{ u\;|\; x\leq u\leq y\;{\rm and}\; e_{x,u},e_{u,y}\in
\mathcal{B}\}$$ and define the relation $\sim$ on $U_{x,y}$ by
$u\sim v$ if and only if there exist a positive integer $n$, and
$u_0=u,u_1,\ldots,u_n=v$ and $z_1,\ldots,z_n$ in $U_{x,y}$, such
that $z_i\leq u_{i-1}$ and $z_i\leq u_i$ for any $1\leq i\leq n$. It
is easy to see that $\sim$ is an equivalence relation on $U_{x,y}$.
Let $U_{x,y}/\sim$ be the associated set of equivalence classes, and
denote by $(U_{x,y}/\sim)_0$ the set of all equivalence classes
$\mathcal{C}$ satisfying the following two
conditions.\\
$\bullet$ If $u\in \mathcal{C}$, and $v\in X$ satisfies $v\leq u$
and $e_{v,y}\in \mathcal{B}$, then $x\leq v$.\\
$\bullet$ If $u\in \mathcal{C}$, and $v\in X$ satisfies $u\leq v$
and $e_{x,v}\in \mathcal{B}$, then $v\leq y$.

\vspace{.4cm}

Now we can describe the balanced bilinear forms on $C$.

\begin{theorem} \label{formebilinc}
A bilinear form $\beta:C\times C\rightarrow K$ is $C^*$-balanced if
and only if there is a family of scalars
$(\alpha_{\mathcal{C}})_{\mathcal{C}\in \bigsqcup\limits_{(x,y)\in
\mathcal{D}}(U_{x,y}/\sim)_0}$ such that for any $e_{t,y},e_{x,z}\in
\mathcal{B}$
\begin{eqnarray*}
\beta(e_{t,y},e_{x,z}) & = & \left\{
\begin{array}{l}
    \alpha_{\mathcal{C}},\,\,\,{\rm if\,}(x,y)\in\mathcal{D}, z=t\in U_{x,y}{\rm\,and\,}{\rm the\,}{\rm class\,}\\
    \hspace{.8cm}\mathcal{C}{\rm\, of\,}z{\rm\, in\,}U_{x,y}/\sim {\rm\, is\,}{\rm\, in\,}(U_{x,y}/\sim)_0\\
    0,\,\,\,\;\;{\rm otherwise}
\end{array}\right.
\end{eqnarray*}
In particular the set of all $C^*$-balanced bilinear forms on $C$ is
in bijective correspondence to $K^{\bigsqcup\limits_{(x,y)\in
\mathcal{D}}(U_{x,y}/\sim)_0}$.
\end{theorem}
\begin{proof}
Assume that $\beta$ is $C^*$-balanced. Fix some $x\leq y$ such that
$U_{x,y}\neq \emptyset$. We first note that if $x\leq z\leq t\leq y$
and $z,t\in U_{x,y}$, then by applying (\ref{eq.d1}) for $p=e_{z,y},
q=e_{x,t}$ and $p_1=q_2=e_{z,t}$, we find that
$\beta(e_{t,y},e_{x,t})=\beta(e_{z,y},e_{x,z})$. Now let $u,v\in
U_{x,y}$ such that $u\sim v$. Let $u_0=u,u_1,\ldots,u_n=v$ and
$z_1,\ldots,z_n$ in $U_{x,y}$, such that $z_i\leq u_{i-1}$ and
$z_i\leq u_i$ for any $1\leq i\leq n$. By the above
$\beta(e_{u_{i-1},y},e_{x,u_{i-1}})=\beta(e_{u_{i},y},e_{x,u_{i}})=\beta(e_{z_{i},y},e_{x,z_{i}})$
for any $i$, and this implies that
$\beta(e_{u,y},e_{x,u})=\beta(e_{v,y},e_{x,v})$. This shows that
$\beta(e_{u,y},e_{x,u})$ takes the same value for any $u$ in the
same
equivalence class in $U_{x,y}/\sim$.\\
Now assume that for some $u\in U_{x,y}$ there is $v\in X$, such that
$v\leq u$, $x\nleqslant v$ and $e_{v,y}\in \mathcal{B}$. Use
(\ref{eq.d2}) for $p=e_{v,y}$, $q=e_{x,u}$ and $p_1=e_{v,u}$. Note
that $p_1\neq q_2$ for any $q_2$. We get that
$\beta(e_{u,y},e_{x,u})=0$. \\
Similarly, if $u\in U_{x,y}$, and there is $v\in X$ such that $u\leq
v$, $v\nleqslant y$ and $e_{x,v}\in \mathcal{B}$, then using
(\ref{eq.d3}) for $p=e_{u,y}, q=e_{x,v}$ and $q_2=e_{u,v}$, we find
that $\beta(e_{u,y},e_{x,u})=0$. We have thus showed that
$\beta$ has the desired form. \\
Conversely, assume that $\beta$ has the indicated form. We show that
it satisfies (\ref{eq.d1}), (\ref{eq.d2}) and (\ref{eq.d3}). Let
$p,q\in \mathcal{B}$ such that $p_1=q_2$ for some $p_1$ and $q_2$.
Then $p=e_{z,y},q=e_{x,t}$ and $p_1=q_2=e_{z,t}$ for some $x\leq
z\leq t\leq y$. Clearly $t\sim z$, and let $\mathcal{C}$ be the
equivalence class of $t$ in $U_{x,y}/\sim$. Then
$\beta(p_2,q)=\beta(e_{t,y},e_{x,t})$ and
$\beta(p,q_1)=\beta(e_{z,y},e_{x,z})$, and they are both equal to
$\alpha_{\mathcal{C}}$ if $\mathcal{C}\in (U_{x,y}/\sim)_0$, and to
0 if $\mathcal{C}\notin (U_{x,y}/\sim)_0$. Thus (\ref{eq.d1})
is satisfied. \\
Let now $p=e_{z,y},p_1=e_{z,t},p_2=e_{t,y}$ and $q=e_{x,u}$ such
that $p_1\neq q_2$ for any $q_2$. Then
$\beta(p_2,q)=\beta(e_{t,y},e_{x,u})$. If $u\neq t$, this is clearly
0. Let $u=t$. Then $x\nleqslant z$, otherwise $p_1=q_2$ for some
$q_2$. We have that $t\in U_{x,y}$, but the equivalence class of $t$
in $U_{x,y}/\sim$ is not in $(U_{x,y}/\sim)_0$, since $e_{z,y}\in
\mathcal{B}$, $z\leq t$, but $x\nleqslant z$. It follows that
$\beta(p_2,q)=0$, and (\ref{eq.d2}) holds. Similarly we can show
that (\ref{eq.d3}) holds.
\end{proof}

\section{Left  quasi-co-Frobenius path subcoalgebras and
subcoalgebras of incidence coalgebras}\label{s3}

In this section we investigate when a path subcoalgebra of a path
coalgebra or a subcoalgebra of an incidence coalgebra is left
co-Frobenius. We keep the notation of Section \ref{sectionbilforms}.
Thus $C$ will be either a path subcoalgebra of a path coalgebra
$K\Gamma$, or a subcoalgebra of an incidence coalgebra $KX$. The
distinguished basis of $C$ consisting of paths or segments will be
denoted by $\mathcal{B}$. We note that in each of the two cases
$\mathcal{B}\cap C_n$ is a basis of $C_n$, where $C_0\subseteq
C_1\subseteq \ldots $ is the coradical filtration of $C$. The
injective envelopes of the simple left (right) comodules were
described in \cite[Lemma 5.1]{simson} for incidence coalgebras and
in \cite[Corollary 6.3]{chin} for path coalgebras. It is easy to see
that these descriptions extend to the following.

\begin{proposition}\label{inj}
(i) If $C$ is a path subcoalgebra, then for each vertex $v$ of
$\Gamma$ such that $v\in C$, the injective envelope of the left
(right) $C$-comodule $Kv$ is (the $K$-span)
$E_l(Kv)=<p\in \mathcal{B}|t(p)=v>$ (and $E_r(Kv)=<p\in \mathcal{B}|s(p)=v>$ respectively).\\
(ii) If $C$ is a subcoalgebra of the incidence coalgebra  $KX$, then
for any $a\in X$ such that $e_{a,a}\in C$, the injective envelope of
the left (right) $C$-comodule $Ke_{a,a}$ is (the $K$-span)
$E_l(Ke_{a,a})=<e_{x,a}|x\in X, e_{x,a}\in C>$ (and
$E_r(Ke_{a,a})=<e_{a,x}|x\in X, e_{a,x}\in C>$).
\end{proposition}

The following shows that we have a good left-right duality for
comodules generated by elements of the basis $\mathcal{B}$.

\begin{lemma}\label{l.duals}
(i) Let $C$ be a subcoalgebra of the incidence coalgebra $KX$, and
let $e_{a,b}\in C$. Then
$(C^*e_{a,b})^*\cong e_{a,b} C^*$ as right $C^*$-modules (or left $C$-comodules).\\
(ii) Let $C$ be a path subcoalgebra of $K\Gamma$, and let $p$ be a
path in $C$. Then $(C^* p)^*\cong p C^*$ as right $C^*$-modules (or
left $C$-comodules).
\end{lemma}
\begin{proof}
(i) Clearly the set of all segments $e_{a,x}$ with $a\leq x\leq b$
is a basis of $C^*e_{a,b}$. Denote by $e_{a,x}^*$ the corresponding
elements of the dual basis of $(C^*e_{a,b})^*$. Since for $c^*\in
C^*$ and $a\leq x,y\leq b$ we have
\begin{eqnarray*}
(e_{a,x}^*c^*)(e_{a,y})&=&\sum_{a\leq z\leq y}c^*(e_{z,y})e_{a,x}^*(e_{a,z})\\
&=&\left\{
\begin{array}{l}
0,\mbox{ if } x\nleqslant y\\ c^*(e_{x,y}), \mbox{ if } x\leq y
\end{array}\right.
\end{eqnarray*}
we get that
\begin{equation} \label{eq1dual}
e_{a,x}^*c^*=\sum_{x\leq y\leq b}c^*(e_{x,y})e_{a,y}^*
\end{equation} On the other hand $e_{a,b}C^*$ has a basis consisting
of all segments $e_{x,b}$ with $a\leq x\leq b$, and
\begin{equation} \label{eq2dual}
e_{x,b}c^*=\sum_{x\leq y\leq b}c^*(e_{x,y})e_{y,b} \end{equation}
Equations (\ref{eq1dual}) and (\ref{eq2dual}) show that the linear
map $\phi:(C^*e_{a,b})^*\ra e_{a,b}C^*$ defined by $\phi
(e_{a,x}^*)=e_{x,b}$, is
an isomorphism of right $C^*$-modules.\\
(ii) Let $p=a_1\ldots a_n$ and $v=s(p)$. Denote $p_i=a_1\ldots a_i$
for any $1\leq i\leq n$, and $p_0=v$. Then $\{ p_0,p_1,\ldots
,p_n\}$ is a basis of $C^*p$, and let $(p_i^*)_{0\leq i\leq n}$ be
the dual basis of $(C^*p)^*$. For any $0\leq t\leq j\leq n$ denote
by $\overline{p_{t,j}}$ the path such that
$p_j=p_t\overline{p_{t,j}}$. Then a simple computation shows that
$p_i^*c^*=\sum_{i\leq j\leq n}c^*(\overline{p_{i,j}})p_j^*$ for any
$i$ and any $c^*\in C^*$.\\
On the other hand, $\{ \overline{p_{i,n}}\;|\; 0\leq i\leq n\}$ is a
basis of $pC^*$, and it is easy to see that
$\overline{p_{i,n}}c^*=\sum_{i\leq r\leq
n}c^*(\overline{p_{i,r}})\overline{p_{r,n}}$ for any $i$ and any
$c^*\in C^*$. Then the linear map $\phi:(C^*p)^*\ra pC^*$ defined by
$\phi (p_i^*)=\overline{p_{i,n}}$ for any $0\leq i\leq n$, is an
isomorphism of right $C^*$-modules.
\end{proof}

For a path subcoalgebra $C$ let us denote by $R(C)$ the set of
vertices $v$ in $C$ such that the set $\{p\in C\; |\; p \mbox{ path
and }s(p)=v\}$ is finite (i.e. $E_r(Kv)$ is finite dimensional) and
contains a unique maximal path. Note that $v\in R(C)$ if and only if
 $E_r(Kv)$ is finite dimensional and local. Indeed, if $E_r(Kv)$
is finite dimensional and contains a unique maximal path
$p=a_1\ldots a_n$, then keeping the notation from the proof of Lemma
\ref{l.duals}, we have that $E_r(Kv)=C^*p$ and
$C^*p_{n-1}=<p_0,\ldots ,p_{n-1}>$ is the unique maximal
$C^*$-submodule of $C^*p$. Conversely, if $E_r(Kv)$ is finite
dimensional and local with the unique maximal subcomodule $N$, then
 the set $(\mathcal{B}\cap E_r(Kv))/N$ is nonempty. If $p$ is a path
which belongs to this set, $E_r(Kv)=C^*p$. Then clearly $p$ is the
unique maximal path in $\{q\in C\; |\; q \mbox{ path and
}s(q)=v\}$.\\
 Similarly, denote by $L(C)$ the set of vertices
$v$ of $C$ such that $E_l(Kv)$ is a finite dimensional local left
$C$-comodule. Also, for each vertex $v\in R(C)$ let $r(v)$ denote
the endpoint of the
maximal path starting at $v$, and for $v\in L(C)$ let $l(v)$ be the starting point of the maximal path ending at $v$.\\
Similarly, for a subcoalgebra $C$ of the incidence coalgebra $KX$,
let $R(C)$ be the set of all $a\in X$ for which $e_{a,a}\in C$ and
the set $\{x\in X\;|\; a\leq x,\,e_{a,x}\in C\}$ is finite and has a
unique maximal element, and $L(C)$ be the set of all $a\in X$ for
which $e_{a,a}\in C$ and the set $\{x\in X|x\leq a,\,e_{x,a}\in C\}$
is finite and has a unique minimal element. As before, $R(C)$
(respectively $L(C)$) consists of those $a\in X$ for which
$E_r(Ke_{a,a})$ (respectively, $E_l(Ke_{a,a})$) are local, hence
generated by a segment. Here $r(a)=r(e_{a,a})$ for $a\in R(C)$
denotes the maximum element in the set $\{x\;|\; x\geq
a,\,e_{a,x}\in C\}$ and $l(a)$ for $a\in L(C)$ means the minimum of
$\{x\;|\; x\leq a, e_{x,a}\in C\}$.

\begin{proposition}\label{p.qcf}
(I) Let $C$ be a path subcoalgebra of the path coalgebra $K\Gamma$.
Then the following are
equivalent.\\
(a) $C$ is left co-Frobenius.\\
(b) $C$ is left quasi-co-Frobenius.\\
(c) $R(C)$ consists of all vertices belonging to $C$,
$r(R(C))\subseteq L(C)$ and $lr(v)=v$, for any vertex $v$ in $C$. \\
(d) For any path $q\in \mathcal{B}$ there exists a path $p\in
\mathcal{B}$  such that $qp\in \mathcal{F}$ (for $\mathcal{F}$ defined in the previous section).\\
(II) Let $C$ be a subcoalgebra of the incidence coalgebra $KX$.
Then the following are equivalent.\\
(a) $C$ is left co-Frobenius.\\
(b) $C$ is left quasi-co-Frobenius.\\
(c) $R(C)$ consists of all $a\in X$ such that $e_{a,a}\in C$,
$r(R(C))\subseteq L(C)$ and $lr(a)=a$, $\forall\,a\in X$ with
$e_{a,a}\in C$.\\
(d) For any segment $e_{x,z}\in C$ there exists $y\geq z$ such that
$e_{z,y}\in C$ and the class of $z$ in $U_{x,y}/\sim$ lies in
$(U_{x,y}/\sim)_0$.
\end{proposition}
\begin{proof}
$(I)$ (a)$\Rightarrow $(b) is clear. \\
(b)$\Rightarrow$(c) We apply the QcF characterization of \cite{II}
and \cite{IG}. If $C$ is left QcF then for any vertex $v\in C$,
there is a vertex $u\in C$ such that $E_r(Kv)\cong E_l(Ku)^*$. Hence
$E_r(Kv)$ is finite dimensional and local (by \cite[Lemma 1.4]{IF}),
so   $v\in R(C)$ and $E_r(Kv)=C^*p$ for a path $p$ by the discussion
preceding this Proposition. Let $t(p)=w$. Then it is easy to see
that the linear map $\phi:C^*p\ra Kw$ taking $p$ to $w$, and any
other $q$ to 0, is a surjective morphism of left $C^*$-modules.
Since $E_r(Kv)\cong E_l(Ku)^*$, there is a surjective morphism of
left $C^*$-modules $E_l(Ku)^*\ra Kw$, inducing an injective morphism
of right $C^*$-modules $(Kw)^*\ra E_l(Ku)$. Since $(Kw)^*\cong Kw$
as right $C^*$-modules, and the socle of the comodule $E_l(Ku)$ is
$Ku$, we must have $w=u$, and thus $u=r(v)$. By Lemma \ref{l.duals},
$E_l(Ku)\cong E_r(Kv)^*=(C^* p)^*\cong p C^*$, so $E_l(Ku)$ is
generated by $p$, and this shows that $p$ is the unique maximal path
ending at $u$. Hence, $u=r(v)\in L(X)$,
and $l(u)=v$. Thus $l(r(v))=v$. \\
(c)$\Rightarrow$ (d) Let $q\in \mathcal{B}$, and let $v=s(q)$. Since
$v\in R(C)$, there exists a unique maximal path $d$ starting at $v$,
and $d=qp$ for some path $p$. We show that $d\in \mathcal{F}$.
Denote $t(d)=v'$, and let $d=q'p'$ for some paths $q',p'$ in
$\mathcal{B}$. Let $u=t(q')=s(p')$. If there is an arrow $b$ (in
$\Gamma_1$) starting at $u$, such that $q'b\in \mathcal{B}$, then
$q'b$ is a subpath of $d$, since $d$ is the unique maximal path
starting at $v$. It follows that $p'$ starts with $b$. On the other
hand, $v'=r(v)\in L(C)$ and $l(v')=lr(v)=v$, so $d$ is the unique
maximal path in $\mathcal{B}$ ending at $v'$. This shows that if an
arrow $a$ (in $\Gamma_1$) ends at $u$, and $ap'\in \mathcal{B}$,
then $ap'$ is a subpath of $d$, so the last arrow of $q'$ is $a$. We
conclude that
$d\in \mathcal{F}$. \\
(d)$\Rightarrow$(a) Choose a family $(\alpha_d)_{d\in \mathcal{F}}$
of scalars, such that $\alpha_d\neq 0$ for any $d$. Associate a
$C^*$-balanced bilinear form $B$ on $C$ to this family of scalars as
in Theorem \ref{formebilpath}. Then $B$ is right non-degenerate, so
$C$ is left co-Frobenius.

 $(II)$ (a)$\Rightarrow$(b) is clear; (b)$\Rightarrow$(c) is proved as the similar implication in
 $(I)$, with paths
replaced by segments.\\
(c)$\Rightarrow$(d) Let $e_{x,z}\in C$. If $r(x)=y$, then clearly
$z\leq y$ and $e_{x,y}\in \mathcal{B}$, so $U_{x,y}=[x,y]$. Then any
two elements in $U_{x,y}$ are equivalent with respect to $\sim$
(since they are both $\geq x$), so there is precisely one
equivalence class in $U_{x,y}/\sim$, the whole of $U_{x,y}$. We show
that this class lies in $(U_{x,y}/\sim)_0$. Indeed, if $u\in
U_{x,y}$, $v\in X$, $v\leq u$ and $e_{v,y}\in \mathcal{B}$, then
$v\in \{ a|e_{a,y}\in \mathcal{B}\}$, and since $l(y)=l(r(x))=x$, we
must have $x\leq v$. Also, if $u\in U_{x,y}$, $v\in X$, $u\leq v$
and $e_{x,v}\in \mathcal{B}$, then $v\in \{a|e_{x,a}\in
\mathcal{B}\}$. Then $v\leq y$ since $r(x)=y$.\\
(d)$\Rightarrow$(a) follows as the similar implication in
 $(I)$ if we take into account Theorem \ref{formebilinc}.
\end{proof}

As a consequence we obtain the following result, which was proved
for incidence coalgebras in \cite{DNV}.

\begin{corollary} \label{proppathcoFrob}
If $C=K\Gamma$, a path coalgebra, or $C=KX$, an incidence coalgebra, the following are equivalent \\
(i) $C$ is co-semisimple (i.e. $\Gamma$ has no arrows for
$C=K\Gamma$,
and the order relation on $X$ is the equality for $C=KX$).\\
(ii) $C$ is left QcF.\\
(iii) $C$ is left co-Frobenius.\\
(iv) $C$ is right QcF.\\
(v) $C$ is right co-Frobenius.
\end{corollary}

As an immediate consequence we describe the situations where a
finite dimensional path algebra is Frobenius. We note that the path
algebra of a quiver $\Gamma$ (as well as the path coalgebra
$K\Gamma$) has finite dimension if and only if $\Gamma$ has finitely
many vertices and arrows, and there are no cycles.

\begin{corollary}
A finite dimensional path algebra is Frobenius if and only if the
quiver has no arrows.
\end{corollary}
\begin{proof}
It follows from the fact that the dual of a finite dimensional path
algebra is a path coalgebra, and by Corollary \ref{proppathcoFrob}.
\end{proof}

\section{Classification of left co-Frobenius path
subcoalgebras}\label{s4}

Proposition \ref{p.qcf} gives information about the structure of
left co-Frobenius path subcoalgebras. The aim of this section is to
classify these coalgebras. We first use Proposition \ref{p.qcf} to
give some examples of left co-Frobenius path subcoalgebras. These
examples will be the building blocks for the classification.

\begin{example}\label{gn}
Let $\Gamma=\AA_{\infty}$ be the quiver such that $\Gamma_0=\ZZ$ and
there is precisely one arrow from $i$ to $i+1$ for any $i\in \ZZ$.

 $$\AA_{\infty}: \;\;\;\; \dots\longrightarrow \circ^{-1}\longrightarrow
\circ^{0}\longrightarrow \circ^{1}\longrightarrow
\circ^2\longrightarrow\dots$$

For any $k<l$, let $p_{k,l}$ be the (unique) path from the vertex
$k$ to the vertex $l$. Also denote by $p_{k,k}$ the vertex $k$. Let
$r:\ZZ \ra \ZZ$ be a strictly increasing function such that $r(n)>n$
for any $n\in \ZZ$. We consider the path subcoalgebra
$K[\AA_{\infty}, r]$ of $K\AA_{\infty}$ with the basis
\begin{eqnarray*}
\mathcal{B}&=&\bigcup_{n\in \ZZ}\{ p\; |\; p \mbox{ is a path in
}\AA_{\infty}, s(p)=n \;{\rm and}\;{\rm length}(p)\leq r(n)-n\}\\
&=&\{ p_{k,l}\;|\; k,l\in \ZZ \;{\rm and}\; k\leq l\leq r(k)\}
\end{eqnarray*}
Note that $K[\AA_{\infty}, r]$ is indeed a subcoalgebra, since
$$\Delta(p_{k,l})=\sum\limits_{i=k}^lp_{k,i}\otimes p_{i,l}, \,\,\,k\leq l$$
The counit is given by
$$\varepsilon(p_{k,l})=\delta_{k,l}$$
Note that this can also be seen as a subcoalgebra of the incidence
coalgebra of $(\NN,\leq)$,
consisting of the segments $e_{k,l}$ for $k\leq l\leq r(k)$.\\
The construction immediately shows that the maximal path starting
from $n$ is $p_{n,r(n)}$. Note that for each $n\in \ZZ$,
$p_{n,r(n)}$ is the unique maximal path into $r(n)$. If there would
be another longer path $p_{l,r(n)}$ into $r(n)$ in
$K[\AA_{\infty},r]$, then $l<n$. Then, since $p_{l,r(n)}$ is among
the paths in $K[\AA_{\infty},r]$ which start at $l$ we must have
that it is a subpath of $p_{l,r(l)}$, and so $r(l)\geq r(n)$. But
since $l<n$, this contradicts the assumption that $r$ is strictly
increasing. Therefore, we see that the conditions of Proposition
\ref{p.qcf} are satisfied: $p_{n,r(n)}$ is the unique maximal path
in the (finite) set of all paths starting from a vertex $n$, and it
is simultaneously the unique maximal path in the (finite) set of all
paths ending at $r(n)$. Therefore if $l:L(C)={\rm Im}(r)\rightarrow
R(C)$ is the function used in Proposition \ref{p.qcf} for
$C=K[\AA_{\infty},r]$ satisfies $l(r(n))=n$. This means that
$K[\AA_{\infty},r]$ is a left co-Frobenius
coalgebra.\\
$K[\AA_{\infty},r]$ is also right co-Frobenius if and only if there
is a positive integer $s$ such that $r(n)=n+s$ for any $n\in \ZZ$.
Indeed, if $r$ is of such a form, then $K[\AA_{\infty},r]$ is right
co-Frobenius by the right-hand version of Proposition \ref{p.qcf}.
Conversely, assume that $K[\AA_{\infty},r]$ is right co-Frobenius.
If $r$ would not be surjective, let $m\in \ZZ$ which is not in the
image of $r$. Then there is $n\in \ZZ$ such that $r(n)<m<r(n+1)$.
The maximal path ending at $m$ is $p_{n+1,m}$. Indeed, this maximal
path cannot start before $n$ (since then $p_{n,r(n)}$ would be a
subpath of $p_{n,m}$ different from $p_{n,m}$), and $p_{n+1,m}$ is a
path in $K[\AA_{\infty},r]$, as a subpath of $p_{n+1,r(n+1)}$. Hence
$r(l(m))=r(n+1)\neq m$, and then $K[\AA_{\infty},r]$ could not be
right co-Frobenius by the right-hand version of Proposition
\ref{p.qcf}, a contradiction. Thus $r$ must be surjective, and then
it must be of the form $r(n)=n+s$ for any $n$, where $s$ is an
integer. Since $n<r(n)$ for any $n$, we must have $s>0$. For
simplicity we will denote $K[\AA_{\infty},r]$ by $K[\AA_{\infty}|s]$
in the case where $r(n)=n+s$ for any $n\in \ZZ$.
\end{example}

\begin{example}\label{gn0}
Let $\Gamma=\AA_{0,\infty}$ be the subquiver of $\AA_{\infty}$
obtained by deleting all the negative vertices and the arrows
involving them. Thus $\Gamma_0=\NN$, the natural numbers (including
0).
$$\AA_{0,\infty}:\;\;\; \circ^{0}\longrightarrow
\circ^{1}\longrightarrow
\circ^2\longrightarrow\circ^3\longrightarrow \dots$$ We keep the
same notation for $p_{k,l}$ for $0\leq k\leq l$. Let $r:\NN\ra\NN$
be a strictly increasing function with $r(0)>0$ (so then $r(n)>n$
for any $n\in \NN$), and define $K[\AA_{0,\infty},r]$ to be the path
subcoalgebra of $K\AA_{0,\infty}$ with basis $\{ p_{k,l}\;|\; k,l\in
\NN, k\leq l\leq r(k)\}$. With the same arguments as in Example
\ref{gn} we see that $K[\AA_{0,\infty},r]$ is a left co-Frobenius
coalgebra. We note that $l(0)=0$, and then $r(l(0))=r(0)>0$. By a
right-hand version of Proposition \ref{p.qcf}, this shows that
$K[\AA_{0,\infty},r]$ is never right co-Frobenius.
\end{example}

\begin{example}\label{lps}
For any $n\geq 2$ we consider the quiver $\CC_n$, whose vertices are
the elements of $\ZZ_n=\{ \overline{0},\ldots,\overline{n-1}\}$, the
integers modulo $n$, and there is one arrow from $\overline{i}$ to
$\overline{i+1}$ for each $i$.

$
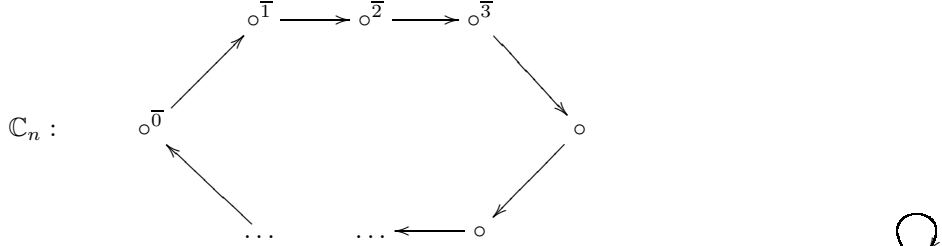
\xymatrix{& &\circ^{\overline{1}}\ar[r] & \circ^{\overline{2}}\ar[r] & \circ^{\overline{3}}\ar[dr] &  \\
 \CC_n:  & \circ^{\overline{0}}\ar[ur] & & & & \circ\ar[dl]\\
& & \dots\ar[ul]& \dots & \circ\ar[l] &}$\\
 We also denote by $\CC_1$ the quiver with one vertex, denoted by $\overline{0}$, and
one arrow $\xymatrix{\circ\ar@(ul,ur)[]}$, and by $\CC_0$ the quiver with one vertex and no arrows.\\
Let $n\geq 1$ and $s>0$ be integers. Let $K[\CC_n,s]$ be the path
subcoalgebra of the path coalgebra  $K\CC_n$, spanned by all paths
of length at most $s$. Denote by $q_{\overline{k}|l}$ the path (in
$\CC_n$) of length $l$ starting at $\overline{k}$, for any
$\overline{k}\in \ZZ_n$ and $0< l\leq s$. Also denote by
$q_{\overline{k}|0}$ the vertex $\overline{k}$. Since the
comultiplication and counit of $K\CC_n$ are given by
$$\Delta(q_{\overline{k}|l})=\sum\limits_{i=0}^lq_{\overline{k}|i}\otimes q_{\overline{k+i}|l-i},$$
$$\varepsilon(q_{\overline{k}|l})=\delta_{0,l}$$
we see that indeed $K[\CC_n,s]=<q_{\overline{k}|l}\;|\;
\overline{k}\in \ZZ, 0\leq l\leq s>$ is a subcoalgebra of $K\CC_n$.
Clearly $q_{\overline{k}|s}$ is the unique maximal path in
$K[\CC_n,s]$ starting at $\overline{k}$, so $\overline{k}\in
R(K[\CC_n,s])$ and $r(\overline{k})=\overline{k+s}$. Also
$\overline{k+s}\in L(K[\CC_n,s])$ and the maximal path ending at
$\overline{k+s}$ is also $q_{\overline{k}|s}$, thus
$lr(\overline{k})=\overline{k}$, and by Proposition \ref{p.qcf} we
get that $K[\CC_n,s]$ is a left co-Frobenius coalgebra. Since it has
finite dimension $n(s+1)$, it is right co-Frobenius, too. This
example was also considered in \cite[1.6]{chyz}.
\end{example}

For a path subcoalgebra $C\subseteq K\Gamma$, denote by $C\cap
\Gamma$ the subgraph of $\Gamma$ consisting of arrows and vertices
of $\Gamma$ belonging to $C$.

\begin{lemma}\label{l.graph}
If $C\subseteq K\Gamma$ is a left co-Frobenius path subcoalgebra,
then $C\cap\Gamma=\bigsqcup\limits_i\Gamma_i$, a disjoint union of
subquivers of $\Gamma$,  where each $\Gamma_i$ is of one of types
$\AA_\infty$, $\AA_{0,\infty}$ or $\CC_n$, $n\geq 0$, and
$C=\bigoplus\limits_{i}C_i$, where $C_i$, a path subcoalgebra of
$K\Gamma_i$, is the subcoalgebra of $C$ spanned by the paths of
$\mathcal{B}$ contained in $\Gamma_i$.
\end{lemma}
\begin{proof}
Let $v$ be a vertex in $C\cap \Gamma$. By Proposition \ref{p.qcf}
there is a unique maximal path $p\in \mathcal{B}$ starting at $v$,
and any path in $\mathcal{B}$ starting at $v$ is a subpath of $p$.
This shows that at most one arrow in $\mathcal{B}$ starts at $v$
(the first arrow of $p$, if $p$ has length $>0$). We show that at
most one arrow in $\mathcal{B}$ ends at $v$, too. Otherwise, if we
assume that two different arrows $a$ and $a'$ in $\mathcal{B}$ end
at $v$, let $s(a)=u$ and $s(a')=u'$ (clearly $u\neq u'$, since at
most one arrow starts at $u$), and let $q$ and $q'$ be the maximal
paths in $\mathcal{B}$ starting at $u$ and $u'$, respectively. Then
$q=az$ and $q'=a'z'$ for some paths $z$ and $z'$ starting at $v$.
But then $z$ and $z'$ are subpaths of $p$, so one of them, say $z$,
is a subpath of the other one. If $w=t(z)$, then $w=r(u)$, so $w\in
L(C)$ and any path in $\mathcal{B}$ ending at $w$ is a subpath of
$q=az$. This provides a contradiction, since $a'z$ is in
$\mathcal{B}$ (as a subpath of $q'$) and ends at $w$,
but it is not a subpath of $q$.\\
We also have that if there is no arrow in $\mathcal{B}$ starting at
a vertex $v$, then there is no arrow in $\mathcal{B}$ ending at $v$
either. Indeed, the maximal path in $\mathcal{B}$ starting at $v$
has length zero, so $r(v)=v$, and then $v\in L(C)$ and $l(v)=v$,
which shows that no arrow in $\mathcal{B}$ ends at
$v$. \\
Now taking the connected components of $C\cap \Gamma$ (regarded just
as an undirected graph), and then considering the (directed) arrows,
we find that $C\cap\Gamma=\bigsqcup\limits_i\Gamma_i$ for some
subquivers $\Gamma_i$ which can be of the types $\AA_\infty$,
$\AA_{0,\infty}$ or $\CC_n$, and this ends the proof.
\end{proof}

\begin{lemma}\label{l.1}
Let $C\subseteq K\Gamma$ be a left co-Frobenius path subcoalgebra.
Let $u,v\in C\cap\Gamma$ be different vertices, and denote by $p_u$
and $p_v$ the maximal paths starting at $u$ and $v$, respectively.
Then $p_u$ is not a subpath of $p_v$.
\end{lemma}
\begin{proof}
Assume otherwise, so $p_u$ is a subpath of $p_v$. We know that $p_u$
and $p_v$ end at $r(u)$ and $r(v)$, respectively. Let $q$ be the
subpath of $p_v$ which starts at $v$ and ends at $r(u)$. Since $p_u$
is a subpath of $p_v$, then $q$ contains $p_u$, too. Then both $q$
and $p_u$ end at $r(u)$, and since by Proposition \ref{p.qcf} $p_u$
is maximal with this property, we get that $q=p_u$. This means that
$u=v$ (as starting points of $p_a$ and $q$), a contradiction.
\end{proof}

Now we are in the position to give the classification result for
left co-Frobenius path subcoalgebras.

\begin{theorem}\label{th.qcf}
Let $C$ be a path subcoalgebra of the path coalgebra $K\Gamma$, and
let $\mathcal{B}$ be a basis of paths of $C$. Then $C$ is left
co-Frobenius if and only if
$C\cap\Gamma=\bigsqcup\limits_i\Gamma_i$, a disjoint union of
subquivers of $\Gamma$ of one of types $\AA_\infty$,
$\AA_{0,\infty}$ or $\CC_n$, $n\geq 1$, and the path subcoalgebra
$C_i$ of $K\Gamma_i$ spanned by the paths of $\mathcal{B}$ contained
in $\Gamma_i$ is of type $K[\AA_{\infty},r]$ if
$\Gamma_i=\AA_{\infty}$, of type $K[\AA_{0,\infty},r]$ if
$\Gamma_i=\AA_{0,\infty}$, of type $K[\CC_n,s]$ with $s\geq 1$ if
$\Gamma_i=\CC_n$, $n\geq 1$, and of type $K$ if $\Gamma_i=\CC_0$. In
this case $C=\bigoplus\limits_{i}C_i$, in particular left
co-Frobenius path subcoalgebras are direct sums of coalgebras of
types $K[\AA_{\infty},r]$, $K[\AA_{0,\infty},r]$, $K[\CC_n,s]$ or
$K$.
\end{theorem}
\begin{proof}
By Lemma \ref{l.graph}, $C\cap \Gamma=\bigsqcup\limits_i\Gamma_i$,
and any $\Gamma_i$ is of one of the types $\AA_\infty$,
$\AA_{0,\infty}$ or $\CC_n$, $n\geq 0$. Moreover,
$C=\bigoplus\limits_{i}C_i$, so $C$ is left co-Frobenius if and only
if all $C_i$'s are left co-Frobenius (see for example \cite[Chapter
3]{DNR}). If all $C_i$'s are of the indicated form, then they are
left co-Frobenius by Examples \ref{gn}, \ref{gn0} and \ref{lps}, and
then so is $C$. Assume now that $C$ is left co-Frobenius. Then each
$C_i$ is left co-Frobenius, so we can reduce to the case where
$\Gamma$ is one of $\AA_\infty$, $\AA_{0,\infty}$ or $\CC_n$, and
$C\cap \Gamma=\Gamma$. As before, for each vertex $v$ we denote by
$r(v)$ the end-point of the unique maximal path in $C$ starting at
$v$, and by $p_v$ this
maximal path. Also denote by $m(v)$ the length of $p_v$. \\
Case I. Let $\Gamma=\CC_n$. If $n=0$, then $C\cong K$. If $n=1$,
then $C\cong K[\CC_1,s]$, since $m(\overline{0})=s>0$ because
$\Gamma_1\subset C$, so there must be at least some nontrivial path
in $C$. If $n\geq 2$, then $m(\overline{k})\leq m(\overline{k+1})$
for any $\overline{k}\in \ZZ_n$, since otherwise
$p_{\overline{k+1}}$ would be a subpath of $p_{\overline{k}}$, a
contradiction by Lemma \ref{l.1}. Thus $m(\overline{0})\leq
m(\overline{1})\leq \ldots m(\overline{n-1})\leq m(\overline{0})$,
so $m(\overline{0})= m(\overline{1})= \ldots m(\overline{n-1})=
m(\overline{0})=s$ for some $s\geq 0$. Since $C\cap \Gamma=\Gamma$,
there are non-trivial
paths in $C$, so $s>0$, and then clearly $C\cong K[\CC_n,s]$. \\
Case II. If $\Gamma=\AA_\infty$ or $\Gamma=\AA_{0,\infty}$, then for
any $n$ (in $\ZZ$ if $\Gamma=\AA_\infty$, or in $\NN$ if
$\Gamma=\AA_{0,\infty}$) $m(n)\leq m(n+1)$ holds, otherwise
$p_{n+1}$ would be a subpath of $p_n$, again a contradiction. Now if
we take $r(n)=n+m(n)$ for any $n$, then $r$ is a strictly increasing
function. Clearly $r(n)>n$, since $m(n)=0$ would contradict $C\cap
\Gamma=\Gamma$. Now it is obvious that $C\cong K[\Gamma,r]$.
\end{proof}

\begin{corollary} \label{clascoFrobenius}
Let $C\subseteq K\Gamma$ be a left and right co-Frobenius path
subcoalgebra. Then $C$ is a direct sum of coalgebras of the type
$K[\AA_{\infty}|s]$, $K[\CC_n,s]$ or $K$.
\end{corollary}
\begin{proof}
It follows directly from Theorem \ref{th.qcf} and the discussion at
the end of each of Examples \ref{gn}, \ref{gn0} and \ref{lps},
concerned to the property of being left and right co-Frobenius.
\end{proof}

\begin{remark}
{\rm (1) We have a uniqueness result for the representation of a
left co-Frobenius path subcoalgebras as a direct sum of coalgebras
of the form $K[\AA_{\infty},r]$, $K[\AA_{0,\infty},r]$, $K[\CC_n,s]$
or $K$. To see this, an easy computation shows that the dual algebra
of a coalgebra of any of these four types does not have non-trivial
central idempotents, so it is indecomposable as an algebra. Now if
$(C_i)_{i\in I}$ and $(D_j)_{j\in J}$ are two families of coalgebras
with indecomposable dual algebras such that $\oplus_{i\in
I}C_i\simeq \oplus _{j\in J}D_j$ as coalgebras, then there is a
bijection $\phi:J\rightarrow I$ such that $D_j\simeq C_{\phi(j)}$
for any $j\in J$. Indeed, if $f:\oplus_{i\in I}C_i\rightarrow \oplus
_{j\in J}D_j$ is a coalgebra isomorphism, then the dual map
$f^*:\prod_{j\in J}D_j^*\rightarrow \prod_{i\in I}C_i^*$ is an
algebra isomorphism. Since all $C_i^*$'s and $D_j^*$'s are
indecomposable, there is a bijection $\phi:J\rightarrow I$ and some
algebra isomorphisms $\gamma_j:D_j^*\rightarrow C_{\phi(j)}^*$ for
any $j\in J$, such that for any $(d_j^*)_{j\in J}\in \prod_{j\in
J}D_j^*$, the map $f^*$ takes $(d_j^*)_{j\in J}$ to the element of
$\prod_{i\in I}C_i^*$ having $\gamma_j(d_j^*)$ on the $\phi(j)$-th
slot. Regarding $C=(C_i)_{i\in I}$ as a left $C^*$-module, and
$D=\oplus _{j\in J}D_j$ as a left $D^*$-module in the usual way,
with actions denoted by $\cdot$, the relation $f(f^*(d^*)\cdot
c)=d^*\cdot f(c)$ holds for any $c\in C$ and $d^*\in D^*$. This
shows that $f$ induces coalgebra isomorphisms
$C_{\phi(j)}\simeq D_j$ for any $j\in J$.\\
(2) The coalgebras of types  $K[\AA_{\infty},r]$,
$K[\AA_{0,\infty},r]$, $K[\CC_n,s]$ or $K$ can be easily classified
if we take into account that the sets of grouplike elements are just
the vertices and the non-trivial skew-primitives are scalar
multiples of the arrows. There are no isomorphic coalgebras of two
different types among these four types. Moreover: (i)
$K[\AA_{\infty},r]\simeq K[\AA_{\infty},r']$ if and only if there is
an integer $h$ such that $r'(n)=r(n+h)$ for any integer $n$; (ii)
$K[\AA_{0,\infty},r]\simeq K[\AA_{0,\infty},r']$ if and only if
$r=r'$; (iii) $K[\CC_n,s]\simeq K[\CC_m,s']$ if and only $n=m$ and
$s=s'$.}
\end{remark}

\section{Examples}  \label{snew}

It is known (see \cite{ni}, \cite{cm}) that any pointed coalgebra
can be embedded in a path coalgebra. Thus it is expected that there
is a large variety of co-Frobenius subcoalgebras of path coalgebras
if we do not restrict only to the class of path subcoalgebras. The
aim of this section is to provide several such examples. We first
explain a simple construction connecting incidence coalgebras and
path coalgebras, and producing examples as we wish.

As a pointed coalgebra,  any incidence coalgebra can be embedded in
a path coalgebra. However, there  is a more simple way to define
such an embedding for incidence coalgebras than for arbitrary
pointed coalgebras. Indeed, let $X$ be a locally finite partially
ordered set. Consider the quiver $\Gamma$ with vertices the elements
of $X$, and such that there is an arrow from $x$ to $y$ if and only
if $x<y$ and there is no element $z$ with $x<z<y$. With this
notation, it is an easy computation to check the following.

\begin{proposition} \label{propembedding}
The linear map $\phi:KX\rightarrow K\Gamma$, defined by
$$\phi(e_{x,y})=\sum_{p\; {\rm path}\atop s(p)=x,t(p)=y}p$$ for any
$x,y\in X, x\leq y$, is an injective coalgebra morphism.
\end{proposition}

Note that in the previous proposition $\phi(KX)$ is in general a
subcoalgebra of $K\Gamma$ which is not a path subcoalgebra. This
suggests that when we deal with left co-Frobenius subcoalgebras of
incidence coalgebras, which of course embed themselves in $K\Gamma$
(usually not as path subcoalgebras),
 structures that are more complicated than those of left co-Frobenius path
subcoalgebras can appear. Thus the classification of left
co-Frobenius subcoalgebras of incidence coalgebras is probably more
difficult. The next example is evidence in this direction.

\begin{example} \label{examplediamond}
Let $s\geq 2$ and $X=\{ a_n|n\in \ZZ\} \cup (\cup_{n\in \ZZ}\{
b_{n,i}|1\leq i\leq s\})$ with the ordering $\leq$ such that
$a_n<b_{n,i}<a_{n+1}$ for any integer $n$ and any $1\leq i\leq s$,
and $b_{n,i}$ and $b_{n,j}$ are not comparable for any $n$ and
$i\neq j$.\\
Let $C$ be the subcoalgebra of $KX$ spanned by the following
elements\\
$\bullet$ the elements $e_{x,x}$, $x\in X$.\\
$\bullet$ all segments $e_{x,y}$ of length 1.\\
$\bullet$ the segments $e_{a_n,a_{n+1}}$, $n\in \ZZ$. \\
$\bullet$ the segments $e_{b_{n,i},b_{n+1,i}}$, with $n\in \ZZ$ and
$1\leq i\leq s$.\\
Then by applying Proposition \ref{p.qcf}, we see that $C$ is co-Frobenius.\\
If we take the subcoalgebra $D$ of $C$ obtained by restricting to
the non-negative part of $X$, i.e. $D$ is spanned by the elements
$e_{x,y}$ in the indicated basis of $C$ with both $x$ and $y$ among
$\{ a_n|n\geq 0\} \cup (\cup_{n\geq 0}\{ b_{n,i}|1\leq i\leq s\})$,
we see that $D$ is left co-Frobenius, but not right co-Frobenius.\\
Now let $\Gamma$ be the quiver associated to the ordered set $X$ as
in the discussion above.
{\small
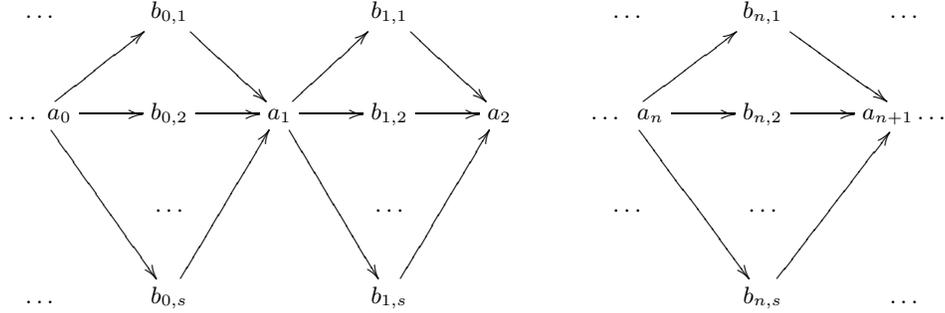
$$
\xymatrix{
 \dots   & b_{0,1} \ar[dr] & & b_{1,1}\ar[dr] & & \dots  & b_{n,1}\ar[dr] & \dots \\
\dots \; a_0\ar[ur]\ar[r]\ar[ddr] & b_{0,2}\ar[r] & a_1\ar[ur]\ar[r]\ar[ddr] & b_{1,2}\ar[r] & a_2 & \dots\;\;  a_n\ar[ur]\ar[r]\ar[ddr] & b_{n,2}\ar[r] & a_{n+1} \dots \\
&  \dots & & \dots & & \dots  & \dots & \\
\dots &  b_{0,s}\ar[uur] & &  b_{1,s}\ar[uur] & &  &
b_{n,s}\ar[uur] & \dots }$$} 

If $\phi:KX\rightarrow K\Gamma$ is the
embedding described in Proposition \ref{propembedding}, then
$\phi(C)$ is a co-Frobenius subcoalgebra of $K\Gamma$. We see that
$\phi(C)$ is the subspace of $K\Gamma$ spanned by  the vertices of
$\Gamma$, the paths of length 1, the paths
$[b_{n,i}a_{n+1}b_{n+1,i}]$ with $n\in \ZZ$ and $1\leq i\leq s$, and
the elements $\sum_{1\leq i\leq s}[a_nb_{n,i}a_{n+1}]$ with $n\in
\ZZ$, thus $\phi(C)$ is not a path subcoalgebra. Here we denoted by
$[b_{n,i}a_{n+1}b_{n+1,i}]$ and $[a_nb_{n,i}a_{n+1}]$ the paths
following the indicated vertices and the arrows between them. By
restricting to the non-negative part of $X$, a similar description
can be given for $\phi(D)$, a subcoalgebra of $K\Gamma$ which is
left co-Frobenius, but not right co-Frobenius.
\end{example}

It is possible to embed some of the co-Frobenius path subcoalgebras
in other path coalgebras as subcoalgebras which are not path
subcoalgebras.

\begin{example}
Consider the quiver $\AA_{\infty}$ with vertices indexed by the
integers, with the path from $i$ to $j$  denoted by $p_{i,j}$.
Consider the path subcoalgebra $D=K[\AA_{\infty}|2]$, with basis $\{
p_{i,i},p_{i,i+1},p_{i,i+2}|i\in \ZZ\}$. We also consider the quiver
$\Gamma$ below
{\small
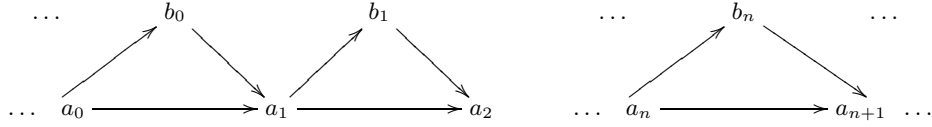
$$
\xymatrix{
\dots  & b_0 \ar[dr] & & b_1\ar[dr] & & \dots   & b_n\ar[dr] & \dots \\
\dots \;\;\;  a_0\ar[ur]\ar[rr] & & a_1\ar[ur]\ar[rr] & & a_2 & \dots \;\;\;
a_n\ar[ur]\ar[rr] & & a_{n+1}\;\; \dots }$$ }

Then $\AA_{\infty}$ is a
subquiver of $\Gamma$ if we identify $a_i$ with $2i$ and $b_i$ with
$2i+1$ for any integer $i$. Thus $K\AA_{\infty}$ is a subcoalgebra
of $K\Gamma$ in the obvious way, and then so is $D$. However, there
is another way to embed $D$ in $K\Gamma$. Indeed, the linear map
$\phi:D\rightarrow K\Gamma$, defined such that
$$\phi(p_{2i,2i})=a_i, \phi(p_{2i+1,2i+1})=b_i,$$
$$\phi(p_{2i,2i+1})=[a_ib_i], \phi(p_{2i+1,2i+2})=[b_ia_{i+1}],$$
$$\phi(p_{2i,2i+2})=[a_ia_{i+1}]+[a_ib_ia_{i+1}],$$
$$\phi(p_{2i+1,2i+3})=[b_ia_{i+1}b_{i+1}]$$
for any $i\in \ZZ$, is an injective morphism of coalgebras. Here we
denoted  by $[a_ib_i]$, $[a_ib_ia_{i+1}]$, etc,  the paths following
the respective vertices and arrows. We conclude that the
subcoalgebra $C=\phi(D)$ of $K\Gamma$, spanned by
 all vertices $a_n,b_n$, all arrows
$[a_na_{n+1}],[a_nb_n],[b_na_{n+1}]$ and the elements
$[a_nb_na_{n+1}]+[a_na_{n+1}]$ and $[b_na_{n+1}b_{n+1}]$, is
co-Frobenius. Note that $D$ is not a path subcoalgebra of $K\Gamma$.
This can be also seen as the subcoalgebra of the incidence coalgebra
of $\ZZ$ with basis consisting of segments of length at most $2$.
\end{example}

Note that in the above example, we can also consider a similar
situation but with all segments $e_{n,n+i}$ of the incidence
coalgebra of $\ZZ$ which have length less or equal to a certain
positive integer $s$ ($i\leq s$); the same properties as above would
then hold for this situation.

\begin{example}
We consider the same situation as above, but we restrict the quiver
$\Gamma$ to  the non-negative part:
{\small 
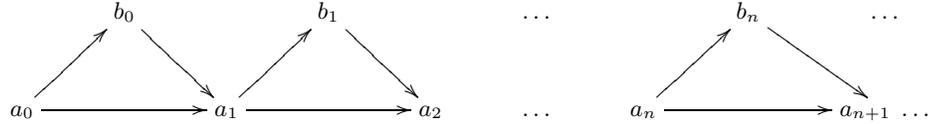
$$
\xymatrix{
  & b_0 \ar[dr] & & b_1\ar[dr] & & \dots & & b_n\ar[dr] & \dots \\
 a_0\ar[ur]\ar[rr] & & a_1\ar[ur]\ar[rr] & & a_2 & \dots & a_n\ar[ur]\ar[rr] & & a_{n+1} \;\dots
}$$ }

Equivalently, we consider the subcoalgebra of the incidence
coalgebra of $\NN$ with a basis of all segments of length less or
equal to $2$ (or $\leq s$ for more generality). This coalgebra is
left co-Frobenius but not right co-Frobenius, it is a subcoalgebra
of an incidence coalgebra, and it can also be regarded as a
subcoalgebra of a path coalgebra, but without a basis of paths.
\end{example}

Now we prove a simple, but useful result, which shows that the
category of incidence coalgebras is closed under tensor product of
coalgebras.

\begin{proposition} \label{tensorincidence}
Let $X,Y$ be locally finite partially ordered sets. Consider on
$X\times Y$ the order $(x,y)\leq(x',y')$ if and only if $x\leq y$
and $x'\leq y'$. Then there is an isomorphism of coalgebras
$K(X\times Y)\cong KX\otimes KY$.
\end{proposition}
\begin{proof}
It is clear that $X\times Y$ is locally finite. We show that the
natural isomorphism of vector spaces $\varphi:K(X\times
Y)\rightarrow KX\otimes KY$,
$\varphi(e_{(x,y),(x',y')})=e_{x,x'}\otimes e_{y,y'}$ is a morphism
of coalgebras. This is well defined by the definition of the order
relation on $X\times Y$. For comultiplication we have
\begin{eqnarray*}
& \sum \varphi(e_{(x,y),(x',y')})_1\otimes (e_{(x,y),(x',y')})_2 = & \\
 = & \sum\limits_{x\leq a\leq x'}\sum\limits_{y\leq b\leq y'}e_{x,a}\otimes e_{y,b}\otimes e_{a,x'}\otimes e_{x',b} & \\
 = & \sum\limits_{(x,y)\leq (a,b)\leq (x',y')}\varphi(e_{(x,y),(a,b)})\otimes \varphi(e_{(a,b),(x',y')}) & \\
 = & \varphi((e_{(x,y),(x',y')})_1)\otimes
\varphi((e_{(x,y),(x',y')})_2) & 
\end{eqnarray*}
and it is also easy to see that $\varepsilon_{KX\otimes
KY}\circ\varphi =\varepsilon_{K(X\times Y)}$.
\end{proof}

\begin{example}
Consider the ordered set $(\ZZ \times \ZZ,\leq)$, with order given
by the direct product of the orders of $(\ZZ,\leq)$ and
$(\ZZ,\leq)$. Thus $(i,j)\leq (p,q)$ if and only if $i\leq p$ and
$j\leq q$. We know from Proposition \ref{tensorincidence} that
$\psi:K\ZZ \otimes K\ZZ\rightarrow K(\ZZ\times \ZZ)$,
$\psi(e_{i,p}\otimes
e_{j,q})=e_{(i,j),(p,q)}$, is an isomorphism of coalgebras.\\
With the notation preceding Proposition \ref{propembedding}, the
quiver $\Gamma$ associated to the locally finite ordered set $(\ZZ
\times \ZZ,\leq)$ is
$$
\xymatrix{
 & \dots & \dots & \dots &\\
\dots \ar[r] & a_{n-1,k+1}\ar[r]\ar[u] & a_{n,k+1}\ar[r]\ar[u] & a_{n+1,k+1}\ar[r]\ar[u] & \dots \\
\dots \ar[r] & a_{n-1,k}\ar[r]\ar[u] & a_{n,k}\ar[r]\ar[u]\ar[r]\ar[u] & a_{n+1,k}\ar[r]\ar[u] & \dots \\
\dots \ar[r] & a_{n-1,k-1}\ar[r]\ar[u] & a_{n,k-1}\ar[r]\ar[u] & a_{n+1,k-1}\ar[r]\ar[u] & \dots \\
& \dots \ar[u] & \dots \ar[u] & \dots \ar[u] & }$$ where we just
denoted the  vertices by $a_{n,k}$ instead of just $(n,k)$. Let
$\phi:K(\ZZ \times \ZZ)\rightarrow K\Gamma$ be the embedding from
Proposition \ref{propembedding}. If we consider the subcoalgebra
$K[\AA_{\infty}|1]$ of $K\ZZ$, then $K[\AA_{\infty}|1]\otimes
K[\AA_{\infty}|1]$ is a subcoalgebra of $K\ZZ\otimes K\ZZ$, so then
$C=\phi \psi(K[\AA_{\infty}|1]\otimes K[\AA_{\infty}|1])$, which is
the subspace spanned by the vertices of $\Gamma$, the arrows of
$\Gamma$, and the elements
$[a_{n,k}a_{n+1,k}a_{n+1,k+1}]+[a_{n,k}a_{n,k+1},a_{n+1,k+1}]$, is a
subcoalgebra of $K\Gamma$. Since $K[\AA_{\infty}|1]$ is
co-Frobenius, and the tensor product of co-Frobenius coalgebras is
co-Frobenius (see \cite[Proposition 4.15]{IG}), we obtain that $C$
is a co-Frobenius coalgebra. Alternatively, it can be seen that
$\psi(K[\AA_{\infty}|1]\otimes K[\AA_{\infty}|1])$, which is the
subspace spanned by the elements $e_{(n,k),(n,k)},
e_{(n,k),(n+1,k)}, e_{(n,k),(n,k+1)}, e_{(n,k),(n+1,k+1)}$ with
arbitrary $n,k\in \ZZ$, is co-Frobenius by applying Proposition
\ref{p.qcf}. $C$ can be seen as both a subcoalgebra of an incidence
coalgebra and of a path coalgebra, but not with a basis of paths. We
note that $C$ is not even isomorphic to a path subcoalgebra. Indeed,
if it were so, it should be isomorphic to some $K[\AA_\infty|s]$,
since it is infinite dimensional and indecomposable. But in $C$, for
any grouplike element $g$ there are precisely two other grouplike
elements $h$ with the property that the set of non-trivial
$(h,g)$-skew-primitive elements is nonempty, while for any grouplike
element g of $K[\AA_\infty|s]$ there is only one such $h$.
\end{example}

With similar arguments, we can give a more general version of the
previous example, by considering finite tensor products of
coalgebras of type $K[\AA_\infty|s]$, as follows.

\begin{example}
Let $D=K[\AA_\infty|s_1]\otimes K[\AA_\infty|s_2]\otimes \ldots
\otimes K[\AA_\infty|s_m]$, where $m\geq 2$ and $s_1,\ldots ,s_m$
are positive integers. Then $D$ is co-Frobenius as a tensor product
of co-Frobenius coalgebras, and $D$ embeds in the $m$-fold tensor
product $K\ZZ\otimes K\ZZ \otimes \ldots \otimes K\ZZ$. But this
last tensor product is isomorphic to the incidence coalgebra of the
ordered set $\ZZ^m=\ZZ\times \ZZ\times \ldots \times \ZZ$, with the
direct product order. The image of $D$ via this embedding is the
subcoalgebra $E$ of $K(\ZZ\times \ZZ\times \ldots \times \ZZ)$
spanned by all the segments $e_{(n_1,\ldots,n_m),(k_1,\ldots,k_m)}$
with $n_1\leq k_1\leq n_1+s_1,\,\ldots \, , n_m\leq k_m\leq n_m+s_m$. \\
Now if we consider the quiver $\Gamma$ associated to the ordered set
$\ZZ\times \ZZ\times \ldots \times \ZZ$ as in the beginning of this
section, we have an embedding of $K(\ZZ\times \ZZ\times \ldots
\times \ZZ)$ in $K\Gamma$. Denote the vertices of $\Gamma$ by
$a_{n_1,\ldots,n_m}$. The image of $E$ through this embedding is the
subcoalgebra $C$ of $K\Gamma$ spanned by all the elements of the
form $S(\Gamma, (n_1,\ldots,n_m),(k_1,\ldots,k_m))$, with
$n_1,\ldots,n_m,k_1,\ldots,k_m$ integers such that $n_1\leq k_1\leq
n_1+s_1,\ldots , n_m\leq k_m\leq n_m+s_m$, where by $S(\Gamma,
(n_1,\ldots,n_m),(k_1,\ldots,k_m))$ we denote the sum of all paths
in $\Gamma$ starting at $a_{n_1,\ldots,n_m}$ and ending at
$a_{k_1,\ldots,k_m}$. Thus $C$ is a co-Frobenius subcoalgebra of
$K\Gamma$, which is also isomorphic to a subcoalgebra of an
incidence coalgebra. However, $C$ is not a path subcoalgebra, and
not even isomorphic to a path subcoalgebra. Indeed, for any
grouplike element $g$ of $E$ there are precisely $m$ grouplike
elements $h$ for which there are non-trivial $(h,g)$-skew-primitive
elements, while in a co-Frobenius path subcoalgebra for any
grouplike element $g$ there is at most one such $h$. \end{example}

\begin{remark}
We note that the co-Frobenius coalgebra $C$ constructed in Example
\ref{examplediamond} is not isomorphic to a coalgebra of the form
$K[\AA_\infty|s_1]\otimes K[\AA_\infty|s_2]\otimes \ldots \otimes
K[\AA_\infty|s_m]$. Indeed, if $g=b_{n,i}$ there exists exactly one
grouplike element $h$ of $C$ such that there are non-trivial
$(h,g)$-skew-primitive elements (this is $h=a_{n+1}$), and if
$g=a_n$ there exist $s$ such grouplike elements $h$ (these are
$b_{n,1},\ldots,b_{n,s}$). On the other hand, in
$K[\AA_\infty|s_1]\otimes K[\AA_\infty|s_2]\otimes \ldots \otimes
K[\AA_\infty|s_m]$ for any grouplike element $g$ there exist
precisely $m$ such elements $h$.
\end{remark}

 We end with another explicit
example, which shows that there are  co-Frobenius subcoalgebras of
path coalgebras that are isomorphic neither to a path subcoalgebra
nor to a subcoalgebra of an incidence coalgebra.

\begin{example}
Let $\Gamma$ be the graph:
$$
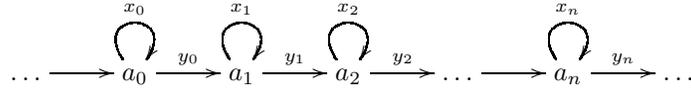
\xymatrix{ \dots\ar[r] & a_0\ar[r]^{y_0}\ar@(ul,ur)[]^{x_0} &
a_1\ar[r]^{y_1}\ar@(ul,ur)[]^{x_1} &
a_2\ar[r]^{y_2}\ar@(ul,ur)[]^{x_2} & \dots \ar[r] &
a_n\ar[r]^{y_n}\ar@(ul,ur)[]^{x_n} &\dots }$$ and let $C$ be the
subcoalgebra of the path coalgebra of $\Gamma$ having a basis the
elements $a_n,x_n,y_n$ and $y_n+x_ny_n$. This is, in fact,
isomorphic to $K[\CC_1|1]\otimes K[\AA_\infty|1]$, so it is
co-Frobenius. By the classification theorem for co-Frobenius path
subcoalgebras and the structure of the skew-primitive elements of
$C$, we see that $C$ is not isomorphic to a path subcoalgebra. We
note that it is not isomorphic either to a subcoalgebra of an
incidence coalgebra, because in an incidence coalgebra, if $g$ is
any grouplike element, there is no $(g,g)$- skew-primitive element,
while in $C$ for each grouplike $g=a_n$, $x_n$ is a $(g,g)$-
skew-primitive.
\end{example}


\section{Hopf algebra structures on path subcoalgebras}\label{s6}
In this section we discuss the possibility of extending the
coalgebra structure of a path subcoalgebra to a Hopf algebra
structure. First of all, it is a simple application of Proposition
\ref{proppathcoFrob} to see when a finite dimensional path coalgebra
has a Hopf algebra structure.

\begin{proposition}
If the path coalgebra $K\Gamma$ is finite dimensional, then it has a
Hopf algebra structure if and only if it is cosemisimple, i.e.
$\Gamma$ has no arrows.
\end{proposition}
\begin{proof}
If the finite dimensional $K\Gamma$ has a Hopf algebra structure,
then it has non-zero integrals, so it is left (and right)
co-Frobenius, and $K\Gamma$ is cosemisimple by Proposition
\ref{proppathcoFrob}. Conversely, if there are no arrows, then
$K\Gamma$ can be endowed with the group Hopf algebra structure
obtained if we consider a group structure on the set of vertices.
\end{proof}

Next, we are interested in finding examples of Hopf algebra
structures that can be defined on some path subcoalgebras. At this
point we discuss only cases where the resulting Hopf algebra has
non-zero integrals, i.e. it is left (or right) co-Frobenius. Thus
the path subcoalgebras that we consider are among the ones in
Corollary \ref{clascoFrobenius}. We ask the following general
question.

{\it {\bf PROBLEM.} Which of the left and right co-Frobenius path
subcoalgebras (classified in Corollary \ref{clascoFrobenius}) can be
endowed with a Hopf algebra structure?}

In the rest of this section we solve the problem in the case where
$K$ is a field containing primitive roots of unity of any positive
order, in particular $K$ has characteristic zero. We will make this
assumption on $K$ from this point on. We just note that some of the
constructions can be also done in positive characteristic, if we
just require that $K$ contains certain primitive roots of unity and
the characteristic of $K$ is large enough.

\begin{proposition}
(I) Let $s>0$ be an integer. Let $q$ be a primitive $(s+1)$th root
of unity in $K$. Let $G$ be a group such that there exist an element
$g\in Z(G)$ of infinite order and a character $\chi\in G^*$ such
that $\chi(g)=q$. Also let $\alpha \in K$ which may be non-zero only
if $\chi^{s+1}=1$. Consider the algebra generated by the elements of
$G$ (and preserving the group multiplication on these elements), and
$x$, subject to relations
$$xh=\chi(h)hx \mbox{ for any }h\in G$$
$$x^{s+1}=\alpha (g^{s+1}-1)$$
(that is, the free or amalgamated product $K[x]*K[G]$, quotient out
by the above relations). Then this algebra has a unique Hopf algebra
structure such that the elements of $G$ are grouplike elements,
$\Delta (x)=1\otimes x+x\otimes g$, and $\varepsilon(x)=0$. We
denote this Hopf algebra
structure by $H_{\infty}(s,q,G,g,\chi,\alpha)$.\\
(II) Let $n\geq 2$ and $s>0$ be integers such that $s+1$ divides
$n$. Let $q$ be a primitive $(s+1)$th root of unity in $K$. Consider
a group $G$ such that there exist an element $g\in Z(G)$ of order
$n$ and a character $\chi\in G^*$ such that $\chi(g)=q$. Also let
$\alpha \in K$ which may be non-zero only if $\chi^{s+1}=1$.
Consider the algebra generated by the elements of $G$ (and
preserving the group multiplication on these elements), and $x$,
subject to relations
$$xh=\chi(h)hx \mbox{ for any }h\in G$$
$$x^{s+1}=\alpha (g^{s+1}-1)$$
Then this algebra has a unique Hopf algebra structure such that the
elements of $G$ are grouplike elements, $\Delta (x)=1\otimes
x+x\otimes g$, and $\varepsilon(x)=0$. We denote this Hopf algebra
structure by $H_n(s,q,G,g,\chi,\alpha)$.
\end{proposition}
\begin{proof}
We consider an approach similar to the one in \cite{bdg}. For both
(I) and (II) we consider the Hopf group algebra $KG$, and its Ore
extension $KG[X,\overline{\chi}]$, where $\overline{\chi}$ is the
algebra automorphism of $KG$ such that $\overline{\chi}(h)=\chi(h)h$
for any $h\in G$. Since $g\in Z(G)$, this Ore extension has a unique
Hopf algebra structure such that $\Delta(X)=1\otimes X+X\otimes g$
and $\varepsilon(X)=0$, by the universal property for Ore extensions
(see for example \cite[Lemma 1.1]{bdg}). Since $(1\otimes
X)(X\otimes g)=q(X\otimes g)(1\otimes X)$, the quantum binomial
formula shows that $\Delta (X^{s+1})=1\otimes X^{s+1}+X^{s+1}\otimes
g^{s+1}$, so then the ideal $I=(X^{s+1}-\alpha (g^{s+1}-1))$ is in
fact a Hopf ideal of $KG[X,\overline{\chi}]$. Then we can consider
the factor Hopf algebra $KG[X,\overline{\chi}]/I$, and this is just
the desired Hopf algebra $H_{\infty}(s,q,G,g,\chi,\alpha)$ in case
(I) and $H_n(s,q,G,g,\chi,\alpha)$ in case (II). The condition that
$\alpha=0$ whenever $\chi^{s+1}\neq 1$ guarantees that the map
$G\rightarrow KG[X,\overline{\chi}]/I$ taking an element $h\in G$ to
its class modulo $I$ is injective, thus $G$ is the group of
grouplike elements of this factor Hopf algebra.
\end{proof}

 In the following example we give examples of co-Frobenius path
subcoalgebras that can be endowed with Hopf algebra structures.
Moreover, we don't only introduce one such structure, but a family
of Hopf algebra structures on each path subcoalgebra considered in
the example.

\begin{example} \label{exemplustructuri}
(i) $K[\AA_{\infty}|s]$ can be endowed with a Hopf algebra structure
for any $s\geq 1$. Indeed, let $q$ be a primitive $(s+1)$th root of
unity in K, and let $\alpha\in K$. We define a multiplication (on
basis elements, then extended linearly) on $K[\AA_{\infty}|s]$ by
$$p_{i,i+u}p_{j,j+v}=\left\{ \begin{array}{l} q^{ju}\,{{u+v}\choose
{u}}_{q}\,p_{i+j,i+j+u+v},\\ \;\hspace{2cm} {\rm if}\; u+v\leq s\\
\alpha
q^{ju}\frac{(u+v-s-1)_q!}{(u)_q!(v)_q!}(p_{i+j+s+1,u+v+i+j}-p_{i+j,u+v+i+j-s-1}), \\ \;\hspace{2cm}{\rm
if}\; u+v\geq s+1
\end{array} \right.$$
where${{u+v}\choose {u}}_{q}$ denotes the $q$-binomial coefficient.
Then this multiplication makes $K[\AA_{\infty}|s]$ an algebra, which
together the initial coalgebra structure define a Hopf algebra
structure on $K[\AA_{\infty}|s]$. Indeed, we can see this by
considering the Hopf algebra
$H_{\infty}(s,q,C_{\infty},c,\chi,\alpha)$, where $C_{\infty}$ is
the (multiplicative) infinite cyclic group generated by an element
$c$, and the character $\chi$ is defined by $\chi(c)=q$. Thus
$H_{\infty}(s,q,C_{\infty},c,\chi,\alpha)$ is generated as an
algebra by the elements $c$ and $x$, subject to relations $xc=q cx$
and $x^{s+1}=\alpha (c^{s+1}-1)$, and with coalgebra structure such
that $\Delta (c)=c\otimes c$, $\varepsilon (c)=1$, and $\Delta
(x)=1\otimes x+x\otimes c$. Since $(1\otimes x)(x\otimes c)=q
(x\otimes c)(1\otimes x)$, we can apply the quantum binomial formula
and get that
$$\Delta (x^u)=\sum_{0\leq h\leq u}{u\choose
h}_{q}x^{u-h}\otimes c^{u-h}x^h$$ and then
$$\Delta \left(\frac{1}{(u)_{q}!}\, c^ix^u\right)=
\sum_{0\leq h\leq u}\frac{1}{(u-h)_{q}!}\, c^ix^{u-h}\otimes
\frac{1}{(h)_{q}!}\, c^{i+u-h}x^h$$ for any $0\leq u\leq s$ and any
integer $i$. Therefore if we denote $\frac{1}{(u)_{q}!}\, c^ix^u$ by
$P_{i,i+u}$, this means that $\Delta (P_{i,i+u})=\sum_{0\leq h\leq
u} P_{i,i+h}\otimes P_{i+h,i+u}$, showing that the linear
isomorphism $\phi:K[\AA_{\infty}|s]\rightarrow
H_{\infty}(s,q,C_{\infty},c,\chi,\alpha)$ taking $p_{i,i+u}$ to
$P_{i,i+u}$ for any $0\leq u\leq s$ and $i\in \ZZ$, is an
isomorphism of coalgebras. Now we just transfer the algebra
structure of $H_{\infty}(s,q,C_{\infty},c,\chi,\alpha)$ through
$\phi^{-1}$ and get precisely the
multiplication formula given above.\\[2mm]
(ii) Let us consider now the coalgebra $C$ which is a direct sum of
a family of copies of (the same) $K[\AA_{\infty}|s]$, indexed by a
non-empty set $P$. Then $C$ can be endowed with a Hopf algebra
structure. To see this, we extend the example from (i) as follows.
Let $G$ be a group such that there exist an element $g\in Z(G)$ of
infinite order and a character $\chi\in G^*$ for which $q=\chi(g)$
is a primitive $(s+1)$th root of unity, and moreover the factor
group $G/<g>$ is in bijection with the set $P$ (note that such a
triple $(G,g,\chi)$ always exists; we can take for instance a group
structure on the set $P$, $G=C_{\infty}\times P$, $g$ a generator of
$C_{\infty}$, and $\chi$ defined such that $\chi(g)=q$ and
$\chi(p)=1$ for any $p\in P$). For simplicity of the notation, we
can assume that $P$ is a set of representatives for the $<g>$-cosets
of G. Consider the Hopf algebra $A=H_{\infty}(s,q,G,g,\chi,\alpha)$,
where $\alpha$ is a scalar which may be non-zero only if
$\chi^{s+1}=1$. Then the subalgebra $B$ of $A$ generated by $g$ and
$x$ is a Hopf subalgebra isomorphic to $K[\AA_{\infty}|s]$ as a
coalgebra, and $A=\oplus_{p\in P}\; pB$ is a direct sum of
subcoalgebras, all isomorphic to $K[\AA_{\infty}|s]$. Thus $A$ is
isomorphic as a coalgebra to $C$,
and we can transfer the Hopf algebra structure of $A$ to $C$.\\[2mm]
(iii) Assume that $n\geq 2$ and $s+1$ divides $n$. Then $K[\CC_n,s]$
can be endowed with a Hopf algebra structure. Indeed, we proceed as
for $K[\AA_{\infty}|s]$, but replacing the Hopf algebra
$H_{\infty}(s,q,C_{\infty},c,\chi,\alpha)$ by
$H_n(s,q,C_n,c,\chi,\alpha)$, where $C_n$ is a cyclic group of order
$n$ with a generator $c$ (we have the same relations for $c$ and $x$
as in (i), to which we add $c^n=1$). Thus the multiplication of
$K[\AA_{\infty}|s]$ is given by
$$q_{\overline{i}|u}q_{\overline{j}|v}=\left\{ \begin{array}{l} q^{ju}\,{{u+v}\choose
{u}}_{q}\,q_{\overline{i+j}|u+v},\\ \hspace{1cm} \;{\rm if}\; u+v\leq s\\
\alpha
q^{ju}\frac{(u+v-s-1)_q!}{(u)_q!(v)_q!}(q_{\overline{i+j+s+1}|u+v-s-1}-q_{\overline{i+j}|u+v-s-1}),\\ \hspace{1cm}\;{\rm
if}\; u+v\geq s+1
\end{array} \right.$$
Also, as in (ii), a direct sum of copies of the same $K[\CC_n,s]$,
indexed by an arbitrary non-empty set $P$, can be endowed with a
Hopf algebra structure isomorphic to some $H_n(s,q,G,g,\chi,\alpha)$
for some $q,G,g,\chi,\alpha$, where $q$ is a primitive $(s+1)$th
root of unity, $G$ is a group, $g\in Z(G)$ is an element of order
$n$, $G/<g>$ is in bijection with $P$, $\chi \in G^*$ is a character
such that $\chi(g)=q$, and $\alpha\in K$ is a scalar which may be
non-zero only if $\chi^{s+1}=1$.\\
The examples given in (iii) appear (for finite sets $P$) in
\cite{chyz}.
\end{example}

Now we can prove the main result of this section.

\begin{theorem} \label{teoremastructuriHopf}
Assume that $K$ is a field containing primitive roots of unity of
any positive order (in particular, $K$ has characteristic $0$). Then
a co-Frobenius path subcoalgebra $C\neq 0$ can be endowed with a
Hopf algebra structure if and only if it is of one of the following
three types:\\
(I) A direct sum of copies (indexed by a set $P$) of the same
$K[\AA_{\infty}|s]$ for some $s\geq 1$. In this case, any Hopf
algebra structure on $C$ is isomorphic to a Hopf algebra of the form
$H_{\infty}(s,q,G,g,\chi,\alpha)$ for some $q,G,g,\chi,\alpha$,
where $q$ is a primitive $(s+1)$th root of unity, $G$ is a group,
$g\in Z(G)$ is an element of infinite order, $G/<g>$ is in bijection
with $P$, $\chi \in G^*$ is a character such that $\chi(g)=q$, and
$\alpha\in K$ is a
scalar which may be non-zero only if $\chi^{s+1}=1$.\\
(II) A direct sum of copies (indexed by a set $P$) of the same
$K[\CC_n,s]$ for some $n\geq 2$ and $s\geq 1$ such that $s+1$
divides $n$. In this case, any Hopf algebra structure on $C$ is
isomorphic to a Hopf algebra of the form $H_n(s,q,G,g,\chi,\alpha)$
for some $q,G,g,\chi,\alpha$, where $q$ is a primitive $(s+1)$th
root of unity, $G$ is a group, $g\in Z(G)$ is an element of order
$n$, $G/<g>$ is in bijection with $P$, $\chi \in G^*$ is a character
such that $\chi(g)=q$, and $\alpha\in K$ is a
scalar which may be non-zero only if $\chi^{s+1}=1$.\\
(III) A direct sum of copies of $K$. In this case, any Hopf algebra
structure on $C$ is isomorphic to a group Hopf algebra $KG$ for some
group $G$.
\end{theorem}
\begin{proof}
By Example \ref{exemplustructuri} we see that a coalgebra of type
(I) or (II) has a Hopf algebra structure. Obviously, a coalgebra of
type (III) is a grouplike coalgebra $KX$ for some set $X$, so then
it has a Hopf algebra structure, obtained if we endow $X$
with a group structure.\\
Conversely, let $C$ be a co-Frobenius path subcoalgebra which can be
endowed with a Hopf algebra structure. By Corollary
\ref{clascoFrobenius}, $C$ is isomorphic to a direct sum of
coalgebras of types $K[\AA_{\infty}|s]$, $K[\CC_n,s]$ or $K$. We
have that $G=G(C)$, the set of all vertices of $C$, is a group with
the induced multiplication. We look at the identity element 1 of
this group and
distinguish three cases.\\
{\it Case 1.} If 1 is a vertex in a connected component of type
$K[\AA_{\infty}|s]$, denote the vertices of this connected component
by $(v_n)_{n\in \ZZ}$ such that $v_0=1$. Also denote by $a_n$ the
arrow from $v_n$ to $v_{n+1}$ for any $n\in \ZZ$. If $g=v_1$, then
$\Delta (a_1)=1\otimes a_1+a_1\otimes g$, and $a_1\notin C_0$. Then
$\Delta (ga_1)=g\otimes ga_1+ga_1\otimes g^2$, and $ga_1\notin C_0$,
so $P_{g^2,g}(C)\nsubseteq C_0$. Since the only $h\in G$ such that
$P_{h,g}(C)$ is not trivial (i.e. $\neq K(h-g)$, or equivalently,
not contained in $C_0$) is $h=v_2$, we obtain that $v_2=g^2$.
Recurrently we see that $v_n=g^n$ for any positive integer $n$, and
also for any negative
integer $n$.\\
Let us take some $h\in G$. Then $\Delta (ha_1)=h\otimes
ha_1+ha_1\otimes hg$ and $ha_1\notin C_0$, so $P_{hg,h}(C)\neq
K(hg-h)$. Hence there is an arrow  starting at $h$ and ending at
$hg$ in $C$; as before, inductively we get that there are in $C$
arrows as follows
$$ \;\;\;\; ^{\hspace{-8mm}\ldots\;\longrightarrow}
_{hg^-1}\hspace{-5mm}^{\circ\;\;\longrightarrow\;\;}
\hspace{0.1mm}^{\circ\;\;\longrightarrow}_h
\hspace{1mm}^{\circ\;\longrightarrow\;}_{hg}
\hspace{0.5mm}^{\circ\;\longrightarrow\;\dots}_{\hspace{-1mm}hg^2}$$
which shows that the vertex $h$ belongs to a connected component $D$
of type $K[\AA_{\infty}|s']$ for some $s'\geq 1$. Moreover, $\Delta
(a_1h)=h\otimes a_1h+a_1h\otimes gh$, we also have $P_{gh,h}(C)\neq
K(gh-h)$, so there is an arrow from $h$ to $gh$ in $C$. This shows
that
$hg=gh$, so then $g$ must lie in $Z(G)$. \\
If we denote by $p_{h,g^ih}$ the unique path from $h$ to $g^ih$, for
any $h\in G$ and $i\geq 0$, then
$$\Delta (p_{1,g^s})-1\otimes p_{1,g^s}-p_{1,g^s}\otimes g^s\in
C_{s-1}\otimes C_{s-1}$$ and $p_{1,g^s}\notin C_{s-1}$. Then
$$\Delta (hp_{1,g^s})-h\otimes hp_{1,g^s}-hp_{1,g^s}\otimes hg^s\in
C_{s-1}\otimes C_{s-1}$$ and $hp_{1,g^s}\notin C_{s-1}$. But it is
easy to check that in the path coalgebra $K\Gamma$ (whose
subcoalgebra is $C$) the relation $\Delta(c)-h\otimes c-c\otimes
hg^s\in (K\Gamma)_{s-1}\otimes (K\Gamma)_{s-1}$ holds if and only if
$c\in (K\Gamma)_{s-1}+Kp_{h,hg^s}$. Applying this for
$c=hp_{1,g^s}\notin C_{s-1}$, we obtain that $hp_{1,g^s}=c'+\gamma
p_{h,hg^s}$ for some $c'\in (K\Gamma)_{s-1}$ and $\gamma \in K^*$.
This shows that $p_{h,hg^s}$ must be in $C$, so it also lies in $D$,
 which implies that $s'\geq s$ (otherwise $D$ cannot have
paths of length $s$). \\
Similarly, since
$$\Delta (h^{-1}p_{h,hg^{s'}})-1\otimes h^{-1}p_{h,hg^{s'}}-h^{-1}p_{h,hg^{s'}}\otimes g^{s'}\in
C_{s'-1}\otimes C_{s'-1}$$ and $h^{-1}p_{h,hg^{s'}}\notin C_{s'-1}$,
we obtain that $s\geq s'$. In conclusion $s'=s$, and $C$ is a direct
sum of coalgebras isomorphic to
$K[\AA_{\infty}|s]$. Moreover, this direct sum is indexed by a set in bijection with $G/<g>$. \\
In order to uncover the Hopf algebra structures on $C$, we use the
Lifting Method proposed in \cite{as}. Since $C_0=K\Gamma$ is a Hopf
subalgebra of $C$, the coradical filtration $C_0\subseteq
C_1\subseteq \ldots$ of $C$ is a Hopf algebra filtration, and we can
consider the associated graded space ${\rm gr}\, C=C_0\oplus
\frac{C_1}{C_0}\oplus \ldots$, which has a graded Hopf algebra
structure. Denote $H=K\Gamma$, the degree 0 component of ${\rm gr}\,
C$, and by $\gamma:H\rightarrow {\rm gr}\, C$ the inclusion
morphism. The natural projection $\pi:{\rm gr}\, C\rightarrow H$ is
a Hopf algebra morphism. Then the coinvariants $R=({\rm gr}\,
C)^{co\, H}$ with respect to the right $H$-coaction induced via
$\pi$, i.e.
$$R=\{ z\in {\rm gr}\, C\;|\; (I\otimes \pi)\Delta (z)=z\otimes
1\}$$ is a left Yetter-Drinfeld module over $H$, with left
$H$-action defined by $h\cdot r=\sum \gamma(h_1)rS(\gamma(h_2))$ for
any $h\in H, r\in R$, and left $H$-coaction $\delta (r)=\sum
r_{(-1)}\otimes r_{(0)}=(\pi \otimes I)\Delta(r)$. Moreover, $R$ is
a graded subalgebra of ${\rm gr}\, C$, with grading denoted by
$R=\oplus_{n\geq 0}R(n)$, and it also has a coalgebra structure with
comultiplication $\Delta_R(r)=\sum r^{(1)}\otimes r^{(2)}=\sum
r_1\gamma\pi (S(r_2))\otimes r_3$, and these make $R$ a braided Hopf
algebra in the category $^H_HYD$ of Yetter-Drinfeld modules over
$H$. The Hopf algebra ${\rm gr}\, C$ can be reconstructed from $R$
by bosonization, i.e. ${\rm gr}\, C\simeq R\# H$, the biproduct of
$R$ and $H$. The multiplication of this biproduct is the smash
product given by $(r\# h)(p\# v)=\sum r(h_1\cdot p)\# h_2v$, while
the comultiplication is the smash coproduct $\Delta (r\# h)=\sum
(r^{(1)}\# (r^{(2)})_{(-1)}h_1)\otimes
(r^{(2)})_{(0)}\# h_2$.\\
Since in our case $C_i$ is the span of all paths of length at most
$i$ in $C$, if $z=\hat{c}\in R(n)$, then $c=\sum_i \alpha_ip_i$, a
linear combination of paths $p_i$ of length $i$, and
$\sum_i\alpha_i\widehat{p_i}\otimes
t(p_i)=\sum_i\alpha_i\widehat{p_i}\otimes 1$. Then $\alpha_i=0$ for
any $i$ such that $t(p_i)\neq 1$, showing that $R(i)$ is spanned by
the classes of the paths of length $i$ which end at $1$. We conclude
that $R(i)$ has dimension 1 for any $0\leq i\leq s$, and ${\rm
dim}(R)=s+1$. By \cite[Theorem 3.2]{as} (see also \cite[Proposition
3.4]{cdmm}) $R$ is isomorphic to a quantum line, i.e. $R\simeq
R_q(H,v,\chi)$ for some primitive $(s+1)$'th root of unity $q$, an
element $v\in G$ and a character $\chi\in G^*$ such that
$\chi(v)=q$, and $\chi(h)hv=\chi(h)vh$ for any $h\in G$, i.e. $v\in
Z(G)$ (we use the notation of \cite[Section 2]{cdmm}). As an algebra
we have $R_q(H,v,\chi)=K[y]/(y^{s+1})$, and the coalgebra structure
is such that the elements $d_0=1,d_1=y,
d_2=\frac{y^2}{(2)_q!},\ldots,\frac{y^s}{(s)_q!}$ form a divided
power sequence, i.e. $\Delta (d_i)=\sum_{0\leq j\leq i}d_j\otimes
d_{i-j}$ for any $0\leq i\leq s$. The $H$-action on $R_q(H,v,\chi)$
is such that $h\cdot y=\chi(h)y$ for any $h\in G$,
and the $H$-coaction is such that $y\mapsto v\otimes y$. \\
By \cite[Proposition 3.1]{cdmm}, there exists a
$(1,v)$-skew-primitive $z$ in $C$, which is not in $C_0$, such that
$vz=qzv$, $C$ is generated as an algebra by $z$ and $G$, and the
class $\hat{z}$ in $\frac{C_1}{C_0}$ corresponds to the element
$y\#1$ in $R_q(H,v,\chi)\# H$ via the isomorphism ${\rm gr}\,
C\simeq R_q(H,v,\chi)\# H$. It follows that $v$ must be $g^{-1}$.
Since for $h\in G$ both $zh$ and $hz$ are
$(h,g^{-1}h)$-skew-primitives, we must have $zh=\lambda hz+\beta
(g^{-1}h-h)$ for some scalars $\lambda$ and $\beta$. But
$zhg=(\lambda hz+\beta(g^{-1}h-h))g=q\lambda hgz+\beta (h-hg)$, and
on the other hand $zgh=qgzh=q\lambda ghz+q\beta (h-hg)$, showing
that $\beta=0$. Thus $zh=\lambda hz$, and passing to ${\rm gr}\, C$,
this gives $\hat{z}h=\lambda h\hat{z}$. But in $R_q(H,v,\chi)\# H$
we have that $(1\# h)(y\# 1)=\chi (h) (y\# 1)(1\# h)$, so
$\lambda=\chi(h)$. Therefore $zh=\chi (h)hz$. Replace the generator
$z$ by $x=gz$, which is a $(g,1)$-skew-primitive. By the quantum
binomial formula we see that $\Delta (x^{s+1})=1\otimes
x^{s+1}+x^{s+1}\otimes g^{s+1}$, so then $x^{s+1}=\alpha
(g^{s+1}-1)$ for some scalar $\alpha$. Since
$x^{s+1}h=\chi(h)^{s+1}hx^{s+1}$, we see that if $\chi^{s+1}\neq 1$,
then $\alpha$ must be zero. Now it is clear that $C\simeq
H_{\infty}(s,q^{-1},G,g,\chi,\alpha)$.

\vspace{.4cm}

{\it Case 2.} If 1 is a vertex in a connected component $D$ of type
$K[\CC_1,s]$, with $s\geq 1$, then let $x$ be the arrow from 1 to 1,
which is a primitive element, i.e. $\Delta(x)=x\otimes 1+1\otimes
x$. Then $gx\notin C_0$ and  $\Delta(gx)=gx\otimes g+g\otimes gx$
for any $g\in G$, so there is an arrow from $gx$ to $gx$. This shows
that $C$ must be a direct sum of coalgebras of type $K[\CC_1,s']$
(for possible different values of $s'$). Then looking at
$\Delta(x^i)-x^i\otimes 1-1\otimes x^i$, it is easy to show by
induction that $x^i$ lies in $D$ for any $i\geq 1$. Since $x$ is a
non-zero primitive element, the set $(x^i)_{i\geq 1}$ is linearly
independent, a contradiction to the finite dimensionality
of $D$. Thus this situation cannot occur.\\
If 1 is a vertex in a connected component of type $K[\CC_n,s]$, with
$n\geq 2$, the proof goes as in Case 1, and leads us to the
conclusion that $C$ is a direct sum of coalgebras isomorphic to
$K[\CC_n,s]$, and that $C$ is isomorphic as a Hopf algebra to one of
the form $H_n(s,q,G,g,\chi,\alpha)$. The only difference is that
instead of using the paths $p_{h,g^ih}$, we deal with paths denoted
by $p_{h|l}$, and meaning the path of length $l$ starting at the
vertex $h$. Also, since $\chi(g)=q$, a $(s+1)$'th root of unity, and
$g^n=1$, $s+1$ must divide $n$.

\vspace{.4cm}

{\it Case 3.} If 1 is a vertex in a connected component of type $K$,
then proceeding as in Case 1, we can see that there are no arrows in
$C$, so $C$ is a direct sum of copies of $K$. Thus $C$ is a
grouplike coalgebra, and Hopf algebra structures on $C$ are just
group Hopf algebras.
\end{proof}

\vspace{.4cm}

We note that the above theorem completely classifies finite
dimensional Hopf algebras whose underlying algebras are quotients of
finite dimensional path algebras by ideals generated by paths, or
whose underlying coalgebras are path subcoalgebras. These are the
algebras $KG$, $H_n(s,q,G,g,\chi,\alpha)$ and their duals, because a
finite dimensional Hopf algebra is Frobenius as an algebra and
co-Frobenius as a coalgebra.

\bigskip\bigskip\bigskip

\begin{center}
\sc Acknowledgment
\end{center}
The research of the first and the third authors was supported by
Grant ID-1904, contract 479/13.01.2009 of CNCSIS. For the second
author, this work was supported by the strategic grant
POSDRU/89/1.5/S/58852, Project ``Postdoctoral programe for training
scientific researchers'' cofinanced by the European Social Fund
within the Sectorial Operational Program Human Resources Development
2007-2013. 


\bigskip\bigskip\bigskip



\begin{thebibliography}{99}


\bibitem{af} M. Aguiar and W. Ferrer Santos, Galois connections
for incidence Hopf algebras of partially ordered sets, \textit{Adv.
Math.} {\bf 151}(2000), 71-100.

\bibitem{as} N. Andruskiewitsch and H. J. Schneider, Liftings of quantum linear
spaces and pointed Hopf algebras of order $p^3$, \textit{J. Algebra}
{\bf 209} (1998), 658-691.

\bibitem{artins} M. Artin and W. F. Schelter, Graded
algebras of global dimension 3, \textit{Adv. Math.} {\bf 66} (1987),
171-216.



\bibitem{bdg} M. Beattie, S. D\u{a}sc\u{a}lescu, L.
Gr\"{u}nenfelder, Constructing pointed Hopf algebras by Ore
extensions, \textit{J. Algebra} {\bf 225}(2000), 743-770.

\bibitem{br} M. Beattie and R. Rose, Balanced bilinear forms on
matrix and matrix-like coalgebras, \textit{Comm. Algebra} {\bf 36}
(2008), 1311-1319.


\bibitem{berger} R. Berger, Koszulity for nonquadratic
algebras, \textit{J. Algebra} {\bf 239} (2001), 705-734.

\bibitem{berger2} R. Berger, Non-homogeneous $N$-Koszul
algebras, \textit{Rev. Uni$\rm \acute{o}$n. Mat. Argentina} {\bf 48}
(2007), 53-56.

\bibitem{B} J. Bichon, N-complexes et algebres de Hopf, \textit{C.R.Acad. Sci. Paris}, Ser I 337 (2003), 441-444.

\bibitem{cdmm} C. C\u{a}linescu, S. D\u{a}sc\u{a}lescu, A.
Masuoka, C. Menini, Quantum lines over non-cocommutative Hopf
algebras, \textit{J. Algebra} {\bf 273} (2004), 753-779


\bibitem{chyz} X.-W. Chen, H.-L. Huang, Y. Ye, P. Zhang, Monomial
Hopf algebras, \textit{J. Algebra} {\bf 275} (2004), 212-232.


\bibitem{chin}
W. Chin, M. Kleiner and D. Quinn, Almost split sequences for
comodules, \textit{J. Algebra} {\bf 249} (2002), 1-19.

\bibitem{cm} W. Chin and S. Montgomery, Basic coalgebras, AMS/IP Stud.
Adv. Math. {\bf 4}(1997), 41-47.

\bibitem{cr} C. Cibils and M. Rosso, Hopf quivers, \textit{J.
Algebra} {\bf 254} (2002), 241-251.


\bibitem{cdb} A. Connes and M. Dubois-Violette, Yang-Mills
algebra, \textit{Lett. Math. Phys.} {\bf 61} (2002), 149-158.


\bibitem{DNV}
S. D\u asc\u alescu, C.N\u ast\u asescu,  G. Velicu, Balanced
bilinear forms and Finiteness properties for incidence coalgebras
over a field, \textit{Rev. Uni$\rm \acute{o}$n. Mat. Argentina} {\bf
51} (2010), 13-20.


\bibitem{DNR} S. D\u asc\u alescu, C. N\u ast\u asescu, \c S. Raianu, \textit{Hopf algebras.
An introduction}, Marcel Dekker, New York, 2001.

\bibitem{gs} E. L. Green and O. Solberg, Basic Hopf algebras and
quantum groups, \textit{Math. Z.} {\bf 229} (1998), 45-76.

\bibitem{gmmz} E. L. Green, E. N. Marcos, R. Martinez-Villa and P.
Zhang, $D$-Koszul algebras, \textit{J. Pure Appl. Algebra} {\bf 193}
(2004), 141-162.


\bibitem{IF}
M.C.Iovanov, Co-Frobenius Coalgebras, \textit{J. Algebra} {\bf 303}
(2006), no. 1, 146-153.

\bibitem{II}
M.C.Iovanov, Abstract Integrals in Algebra, preprint,
arXiv:0810.3740.

\bibitem{IG}
M.C.Iovanov, Generalized Frobenius Algebras and the Theory of Hopf
Algebras, preprint arXiv:0803.0775


\bibitem{jr} S. A. Joni and G. C. Rota, Coalgebras and bialgebras in
combinatorics, \textit{Stud. Appl. Math.} {\bf 61} (1979), 93-139.


\bibitem{mo} S. Montgomery,  \textit{Hopf algebras and their actions
on rings}, CBMS Reg. Conf. Series {\bf 82}, American Mathematical
Society, Providence, 1993.


\bibitem{ni} W. D. Nichols, Bialgebras of type one, \textit{Comm.
Algebra} {\bf 6} (1978), 1521-1552.


\bibitem{sch} W. Schmitt, Incidence Hopf algebras, \textit{J. Pure Appl.
Algebra} {\bf 96} (1994), 299-330.

\bibitem{simson}
D. Simson, Incidence coalgebras of intervally finite posets, their
integral quadratic forms and comodule categories, \textit{Colloq.
Math.} {\bf 115} (2009), 259-295.




\end{thebibliography}
\end{document}